\documentclass[a4paper,12pt,reqno]{amsart}

\usepackage{amsmath,amsfonts,amsthm,amssymb,color}
\usepackage[latin1]{inputenc}
\usepackage{xy} 
\usepackage{graphicx}
\usepackage{pdfsync}
\usepackage{pstricks}
\usepackage{stmaryrd}

\newcommand{\der}{\delta}

\newcommand{\cacha}{\Hat{\mathcal{C}}}

\newcommand{\delha}{\hat{\delta}}

\newcommand{\norm}[1]{\lVert #1\rVert}

\newcommand{\yti}{\tilde{y}}

\newcommand{\xti}{\tilde{x}}
\newcommand{\Xti}{\tilde{X}}
\newcommand{\Jti}{\tilde{J}}
\newcommand{\Kti}{\tilde{K}}

\newcommand{\ka}{\kappa}

\newcommand{\xrgh}{\mathbf{x}}
\newcommand{\bt}{\mathbf{2}}

\DeclareMathOperator{\id}{\text{Id}}

\makeatletter
\newcommand{\eqcolon}{\mathrel{\mathord{=}\raise.2\p@\hbox{:}}}
\newcommand{\coloneq}{\mathrel{\raise.2\p@\hbox{:}\mathord{=}}}
\makeatother

\newcommand{\R}{\mathbb R}
\newcommand{\N}{\mathbb N}

\newcommand{\cb}{\mathcal B}
\newcommand{\co}{\mathcal O}
\newcommand{\cac}{\mathcal C}

\newcommand{\cd}{\mathcal D}

\newcommand{\cj}{\mathcal J}
\newcommand{\cl}{\mathcal L}
\newcommand{\cn}{\mathcal N}
\newcommand{\cq}{\mathcal Q}

\newcommand{\cs}{\mathcal S}

\newcommand{\al}{\alpha}
\newcommand{\ep}{\varepsilon}

\newcommand{\ga}{\gamma}
\newcommand{\la}{\lambda}

\newcommand{\si}{\sigma}

\newcommand{\vp}{\varphi}

\newcommand{\be}{\beta}

\newcommand{\lp}{\left(}
\newcommand{\rp}{\right)}
\newcommand{\lc}{\left[}
\newcommand{\rc}{\right]}
\newcommand{\lcl}{\left\{}
\newcommand{\rcl}{\right\}}
\newcommand{\lln}{\left|}
\newcommand{\rrn}{\right|}

\newcommand{\bean}{\begin{eqnarray*}}
\newcommand{\eean}{\end{eqnarray*}}
\newcommand{\ben}{\begin{enumerate}}
\newcommand{\een}{\end{enumerate}}
\newcommand{\beq}{\begin{equation}}
\newcommand{\eeq}{\end{equation}}

\newtheorem{theorem}{Theorem}[section]

\newtheorem{corollary}[theorem]{Corollary}

\newtheorem{definition}[theorem]{Definition}

\newtheorem{lemma}[theorem]{Lemma}

\newtheorem{proposition}[theorem]{Proposition}

\theoremstyle{remark}
\newtheorem{remark}[theorem]{Remark}

\begin{document}

\begin{center}
{\large\textbf{
A discrete approach to Rough Parabolic Equations
}}\\~\\
Aur\'elien Deya\footnote{Institut \'Elie Cartan, Universit\' e 
Henri Poincar\' e, BP 70239, 54506 Vandoeuvre-l\`es-Nancy, France. Email: {\tt Aurelien.Deya@iecn.u-nancy.fr}}.
\end{center}

\bigskip

{\small \noindent {\bf Abstract:} By combining the formalism of \cite{RHE} with a discrete approach close to the considerations of \cite{Davie}, we interpret and we solve the rough partial differential equation $dy_t=A y_t \, dt+\sum_{i=1}^m f_i(y_t) \, dx^i_t$ ($t\in [0,T]$) on a compact domain $\mathcal{O}$ of $\R^n$, where $A$ is a rather general elliptic operator of $L^p(\mathcal{O})$ ($p>1$), $f_i(\vp)(\xi):=f_i(\vp(\xi))$ and $x$ is the generator of a $2$-rough path. The (global) existence, uniqueness and continuity of a solution is established under classical regularity assumptions for $f_i$. Some identification procedures are also provided in order to justify our interpretation of the problem.

\bigskip


\bigskip

 }

\bigskip

\noindent
Submitted to EJP on November 6, 2010. Final version accepted July 8, 2011. 

\section{Introduction}

The rough paths theory introduced by Lyons in \cite{Lyons} and then refined by several authors (see the recent monograph \cite{FVbook} and the references therein) has led to a very deep understanding of the standard rough systems
\begin{equation}\label{eq-std}
dy^i_t=\sum_{j=1}^m \si_{ij}(y_t) \, dx^j_t \quad , \quad y_0=a \in \R^d \ , \ t\in [0,T],
\end{equation}
where $\si_{ij}:\R \to \R$ is a smooth enough vector field and $x:[0,T] \to \R^m$ is a so-called rough path, that is to say a function allowing the construction of iterated integrals (see Assumption (X)$_\ga$ for the definition of a $2$-rough path and \cite{LQbook} for a rough path of any order). The theory provides for instance a new pathwise interpretation of stochastic systems driven by very general Gaussian processes, as well as fruitful and highly non-trivial continuity results for the Itô solution of (\ref{eq-std}), i.e., when $x$ is a standard Brownian motion.

\smallskip

One of the new challenges of the rough paths theory now consists in adapting the machinery to infinite-dimensional (rough) equations that involves a non-bounded operator, with, as a final objective, the possibility of new pathwise interpretations for stochastic PDEs. Some progresses have recently been made in this direction, with on the one hand the viscosity-solution approach due to Friz \textit{et al} (see \cite{car-friz,car-friz-ober,friz-ober-1,diehl-friz}) and on the other hand, the development of a specific algebraic formalism by Gubinelli \textit{et al} (see \cite{GLT,GT,RHE}). 

\smallskip

The present paper is a contribution to this global project. It aims at providing, in a concise and self-contained formulation, the analysis of the following rough evolution equation: 
\begin{equation}\label{eq-intro}
y_0=\psi \in L^p(\co) \quad , \quad dy_t=Ay_t \, dt+\sum_{i=1}^m f_i(y_t) \, dx^i_t \quad , \quad t\in [0,T],
\end{equation}
where $A$ is a rather general elliptic operator on a bounded domain $\co$ of $\R^n$ (see Assumptions (A1)-(A2)), $f_i(\vp)(\xi):=f_i(\vp(\xi))$ and $x$ generates a $m$-dimensional $2$-rough path (see Assumption (X)$_\ga$). Although the global form of (\ref{eq-intro}) is quite similar to the equation treated in \cite{RHE}, several differences and notable improvements justify the interest of our study:
\begin{list}{\labelitemi}{\leftmargin=1em\itemsep=0.5em}
\item[(i)] The equation is here analysed on a compact domain $\co$ of $\R^n$. This allows to simplify the conditions relative to the vector field $f_i$, which reduce to the classical assumptions of rough paths theory, ie $k$-times differentiable ($k \in \N^\ast$) with bounded derivatives (see Assumption (F)$_k$).
\item[(ii)] The conditions on $p$ are less stringent than in \cite{RHE}, where $p$ has to be taken very large. It will here be possible to show the existence and uniqueness of a solution in $L^p(\co)$ (for a smooth enough initial condition $\psi$) as soon as $p>n$ (see Theorem \ref{theo-uni}). In particular, we can go back to the Hilbert framework of \cite{GT} for the one-dimensional equation ($n=1,p=2$).
\item[(iii)] Last but not least, the arguments we are about to use lead to the existence of a \emph{global} solution for (\ref{eq-intro}), defined on any time interval $[0,T]$. This is is a breakthrough with respect to \cite{GT,RHE}, where only local solutions are obtained, on a time interval that depends on the data of the problem, namely $x$, $f$ and $\psi$.
\end{list}

\

In order to reach these three improvements, the strategy will combine elements of the formalism used in \cite{RHE} with a discrete approach of the equation, close to the machinery developped in \cite{Davie} for rough standard systems. A first step consists of course in giving some reasonable sense to Equation (\ref{eq-intro}). We have chosen to work with an interpretation à la Davie, derived from the expansion of the ordinary solution (see Definition \ref{defi-solu}), and we have left aside the sewing map at the core of the constructions in \cite{RHE}. Note however that the expansion under consideration here relies on the operator-valued paths $X^{x,i},X^{ax,i},X^{xx,ij}$ which were identified in \cite{RHE} (see Subsection \ref{subsec:infinite-rough-path}), and which plays the role of an infinite-dimensional rough path adapted to the problem. When applying the whole procedure to a differentiable driving path $x$ (resp. a standard Brownian motion), the solution that we retrieve coincides with the classical solution (resp. the Itô solution), as reported in Subsection \ref{subsec:notion-solu}. Together with the continuity statement of Theorem \ref{theo-conti}, this identification procedure allows to fully justify our interpretation of (\ref{eq-intro}) (see Corollary \ref{coro:sol-rough-path} and Remark \ref{rk:discuss-sol}). 

\smallskip

Once endowed with this interpretation, our solving method is based on a discrete approach of the problem: as in \cite{Davie}, the solution is obtained as the limit of a discrete scheme the mesh of which tends to $0$. Nevertheless, some fundamental differences arise when trying to mimic the strategy of \cite{Davie}. To begin with, the middle-point argument at the root of the reasoning in the diffusion case (see the proof of \cite[Lemma 2.4]{Davie}) cannot take into account the space-time interactions that occur in the study of PDEs, i.e., the classical estimates (\ref{well-known-1}) and (\ref{well-known-2}). Therefore, the argument must here be replaced with a little bit more complex algorithm described in Appendix A, and which will be used throughout the paper. Let us also mention that the expansion of the vector field $f_i(\vp)(\xi):=f_i(\vp(\xi))$ is not as easy to control as in the standard finite-dimensional case, even if one assumes that the functions $f_i:\R \to \R$ are very smooth. Observe for instance that if $W^{\al,p}$ ($\al \in (0,1)$) stands for the fractional Sobolev space likely to accomodate the solution path, and if $f_i$ is assumed to be differentiable, bounded with bounded derivative, then one can only rely on the non-uniform estimate (see \cite{sickel})
$$\norm{f_i(\vp)}_{W^{\al,p}}\leq \norm{f_i}_{L^\infty(\R)}+\norm{f_i'}_{L^\infty(\R)} \norm{\vp}_{W^{\al,p}} \quad \text{for any} \ \vp \in W^{\al,p}.$$
Consequently, more subtle patching arguments must be put forward so as to exhibit a global solution. The strategy involves in particular a careful examination of the dependence on the initial condition at each step of the procedure (see for instance the controls (\ref{contr-sob-j}) and (\ref{contr-b-p-2})). 

\

The paper is structured as follows: In Section \ref{sec:interpretation}, we gather all the elements that allow to understand our interpretation of Equation (\ref{eq-intro}), and we state the three main results of the paper, namely Theorems \ref{theo-exi}-\ref{theo-conti}. The three sections that follow are dedicated to the proof of each of these results, with the existence theorem first (Section \ref{sec:existence}) and then the uniqueness (Section \ref{sec:uni}) and continuity (Section \ref{sec:conti}) results. Finally, Appendix A contains the description and the analysis of the algorithm at the root of our machinery, while Appendix B is meant to provide the details relative to the identification procedure in the Brownian case (see Proposition \ref{prop:cas-brown}).

\

For the sake of clarity, we shall only consider Equation (\ref{eq-intro}) on the generic interval $[0,1]$. It is however easy to realize that the whole reasoning remains valid on any (fixed) finite interval $[0,T]$ at the price of very minor modifications. 

\smallskip

Throughout the paper, we will denote by $\cac^{k,\textbf{b}}(\R;\R^l)$ ($k,l \in \N^\ast$) the set of $\R^l$-valued functions which are $k$-times differentiable with bounded derivatives. 

\smallskip

Finally, we will use the classical convention for the summation over indexes $x^i y_i=\sum_i x^i y_i$, whenever the underlying index set is obvious from the context.

\section{Interpretation of the equation}\label{sec:interpretation}

We first give some precisions about the setting of our study, as far as the operator $A$, the driving path $x$ and the vector field $f_i$ are concerned (Subsection \ref{subsec:assump}). Then we introduce the notation and the tools designed for our analysis (Subsections \ref{subsec:incre} and \ref{subsec:infinite-rough-path}), and which enable us to interpret (\ref{eq-intro}) (Subsection \ref{subsec:notion-solu}). We finally state the three main results of the paper (Subsection \ref{subsec:main-results}), and we discuss some possible extensions of the strategy to rougher driving paths (Subsection \ref{subsec:rougher-path}).

\subsection{Assumptions}\label{subsec:assump}
As it was announced in the introduction, we mean to tackle the equation $dy_t=Ay_t \, dt+f_i(y_t) \, dx^i_t$, $t\in [0,1]$, in $L^p(\mathcal{O})$, where $\mathcal{O}$ is a bounded domain of $\R^n$, $A$ is an elliptic operator, $f_i(\vp)(\xi):=f_i(\vp(\xi))$ and $x$ is a Hölder path. More precisely, to be in a position to interpret and solve this equation, we will be led to assume that (some of) the following conditions are satisfied:

\

\textbf{Assumption (A1):} $A$ generates an analytic semigroup of contraction $S$ on any $L^p(\mathcal{O})$. Under this hypothesis, we will denote $S_{ts}:=S_{t-s}$ ($s\leq t$), $\cb_p:=L^p(\co)$, $\cb_{\al,p}:=\text{Dom}(A_p^\al)$, and we endow the latter space with the graph norm $\norm{\vp}_{\cb_{\al,p}}:=\norm{A_p^\al \vp}_{L^p(\co)}$. We also assume that for any function $g\in \cac^{1,\textbf{b}}(\R;\R)$, there exists a constant $c^1_g$ such that
\begin{equation}\label{contr-nem-1}
\norm{g(\vp)}_{\cb_{1/2,p}} \leq c^1_g \{ 1+\norm{\vp}_{\cb_{1/2,p}}\}
\end{equation}
and for any function $g\in \cac^{2,\textbf{b}}(\R;\R)$, there exists a constant $c^2_g$ such that
\begin{equation}\label{contr-nem-2}
\norm{g(\vp)}_{\cb_{\al,p}} \leq c^2_g \{ 1+\norm{\vp}_{\cb_{\al ,p}}^2\} \quad \text{if} \ \al \in (1/2,1) \ \text{and} \ 2\al p >n,
\end{equation}
where, in (\ref{contr-nem-1}) and (\ref{contr-nem-2}), $g(\vp)$ is just understood in the composition sense, i.e., $g(\vp)(\xi):=g(\vp(\xi))$.

\

\textbf{Assumption (A2):} If $2 \al p>n$, then $\cb_{\al,p}$ is a Banach algebra continuously included in the space $\cb_\infty$ of continuous functions on $\overline{\mathcal{O}}$.

\

\textbf{Assumption (X)$_{\ga}$:} $x$ allows the construction of a $2$-rough path 
$$(x,\xrgh^\bt)\in \cac_1^\ga([0,1];\R^m) \times \cac_2^{2\ga}([0,1];\R^{m,m})$$
for some (fixed) coefficient $\ga \in  (1/3,1/2)$. In other words, we assume that $x$ is a $\ga$-Hölder path and that there exists a 2-variable path $\xrgh^\bt$ (also called a Lévy area) such that for any $0\leq s \leq u \leq t \leq 1$,
$$\lln \xrgh^{\bt}_{ts} \rrn \leq c \lln t-s \rrn^{2\ga} \quad \text{and} \quad \xrgh^{\bt,ij}_{ts}-\xrgh^{\bt,ij}_{tu}-\xrgh^{\bt,ij}_{us}=(x^i_t-x^i_u)(x^j_u-x^j_s).$$
We will then denote $$\norm{\xrgh}_\ga:=\cn[x;\cac_1^\ga([0,1];\R^m)]+\cn[\xrgh^\bt;\cac_2^{2\ga}([0,1];\R^{m,m})],$$
where
$$\cn[x;\cac_1^\ga([0,1];\R^m)]:=\sup_{0\leq s <t \leq 1} \frac{\lln x_t-x_s \rrn}{\lln t-s \rrn^\ga} \ , \ \cn[\xrgh^\bt;\cac_2^{2\ga}([0,1];\R^{m,m})]:=\sup_{0\leq s <t \leq 1} \frac{\lln \xrgh^\bt_{ts}\rrn}{\lln t-s \rrn^{2\ga}}.$$

\

\textbf{Assumption (F)$_k$:} $f$ belongs to $\cac^{k,\textbf{b}}(\R;\R^m)$.

\

Before pondering over the plausibility of these conditions, let us precise that we henceforth focus on the mild formulation of Equation (\ref{eq-intro})
\begin{equation}\label{equa-mild}
y_t=S_t\psi +\int_0^t S_{tu} \, dx^i_u \, f_i(y_u) \quad , \quad t\in [0,1].
\end{equation}
This is a standard change of perspective for the study of (stochastic) PDEs (see \cite{daprato-zab}), which allows to use the regularizing properties of the semigroup. In retrospect, owing to the regularity assumptions on $f$, it will however be possible to make a link between the mild and strong interpretations of the equation, see Remark \ref{rk:discuss-sol}.

\

\noindent
\textbf{Application:} Properties (A1)-(A2) are satisfied by any elliptic operator on $L^p((0,1)^n)$ that can be written as
\begin{equation}
A=-\sum_{i,j=1}^n \partial_{\xi_i}(a_{ij} \cdot \partial_{\xi_j})+c \quad , \quad \cd(A_p):=W^{2,p}((0,1)^n) \cap W_0^{1,p}((0,1)^n),
\end{equation}
where $c\geq 0$ and the functional coefficients $a_{ij}$ are bounded, differentiable with bounded derivatives on $[0,1]^n$. Indeed, under these assumptions, it is proven in \cite{davies} that $A$ generates an analytic semigroup of contraction. Then, thanks to \cite{pru-sohr}, one can identify the domain $\cd(A_p^\al)$ with the complex interpolation $[L_p,\cd(A_p)]_{\al}$ and one can use the result of \cite{seeley} to assert that $\norm{.}_{\cd(A_p^\al)} \sim \norm{.}_{F^{2\al}_{p,2}}$, where $F^{2\al}_{p,2}$ is the classical Triebel-Lizorkin space described (for instance) in \cite{run-sick}. The results of \cite{run-sick} (resp. \cite{sickel}) finally enables us to check Condition (A2) (resp. the controls (\ref{contr-nem-1}) and (\ref{contr-nem-2})).

\smallskip

\noindent
As far as Condition (X)$_\ga$ is concerned, the process that we have in mind in this paper is the fractional Brownian motion $B^H$ with Hurst index $H>1/3$, for which the (a.s) existence of a Lévy area has been established in \cite{CQ}. Condition (X)$_\ga$ is in fact satisfied by a larger class of Gaussian processes, as reported in \cite{FVbook}.

\smallskip
\noindent
In brief, under the above-stated regularity assumptions, the results that we are about to state and prove can be applied to the stochastic equation
$$dY_t=\lc -\sum_{i,j=1}^n \partial_{\xi_i}(a_{ij} \cdot \partial_{\xi_j}Y_t)+c Y_t \rc dt+\sum_{i=1}^m f_i(Y_t) \, dB^{H,i}_t \quad , \quad t\in [0,1] \ , \ \xi\in (0,1)^n.$$

\

\subsection{Hölder spaces}\label{subsec:incre}

We suppose in this subsection that Assumption (A1) is satisfied. In order to introduce the functional framework of our analysis, let us focus on the following consideration: we know that one of the most appropriate space for the study of rough standard systems is the set of Hölder paths $\{y:[0,1]\to \R^d: \ \lln y_t-y_s \rrn \leq c \lln t-s \rrn^\ga \}$ (see \cite{gubi}), and this is (among others) due to the convenient expression for the variations of the solution $y$ of (\ref{eq-std}), namely $y_t-y_s=\int_s^t \si_{ij}(y_u) \, dx^j_u$. Here, if we denote by $y$ the solution of (\ref{equa-mild}) (assume for the moment that $x$ is a differentiable path), it is readily checked that for all $s<t$,
$$y_t-y_s=\int_s^t S_{tu} \, dx^i_u \, f_i(y_u)+a_{ts} y_s , \quad \text{where} \ a_{ts}:=S_{ts}-\id.$$
With this observation in mind, the following notation arises quite naturally:

\smallskip

\noindent
\textbf{Notation.} For all paths $y:[0,1] \to \cb_p$ and $z:\cs_2\to \cb_p$, where $\cs_2:=\{(t,s)\in [0,1]^2: \ s\leq t\}$, we set, for $s\leq u\leq t \in [0,1]$,
\begin{equation}\label{incr-std}
(\der y)_{ts}:=y_t-y_s \quad , \quad (\delha y)_{ts}:=(\der y)_{ts}-a_{ts}y_s=y_t-S_{ts}y_s,
\end{equation}
\begin{equation}\label{incr-twis}
(\delha z)_{tus}:=z_{ts}-z_{tu}-S_{tu}z_{us}.
\end{equation}

\smallskip

\noindent
The (ordinary) system (\ref{equa-mild}) can now be written in the convenient form
\begin{equation}\label{eq-base}
y_0=\psi \quad , \quad (\delha y)_{ts}=\int_s^t S_{tu} \, dx^i_u \, f_i(y_u) \quad , \quad s,t\in [0,1].
\end{equation}
To make the notation (\ref{incr-std})-(\ref{incr-twis}) even more legitimate in this convolutional context, we let the reader observe the following elementary properties:
\begin{proposition}
Let $y:[0,1]\to \cb_p$ and $x:[0,1] \to \R$ be differentiable paths. Then it holds:
\begin{itemize}
\item Telescopic sum: $\delha (\delha y)_{tus}=0$ and $(\delha y)_{ts}=\sum_{i=0}^{n-1} S_{tt_{i+1}}(\delha y)_{t_{i+1}t_i}$ for any partition $\{s=t_0 <t_1 <\ldots <t_n=t\}$ of an interval $[s,t]$ of $[0,1]$.
\item Chasles relation: if $\cj_{ts}:=\int_s^t S_{tu} \, dx_u \, y_u$, then $\delha \cj=0$.
\end{itemize}
\end{proposition}

\

\noindent
Like with the standard finite-dimensional systems, the rough-paths treatment of Equation (\ref{eq-base}) leans on the controlled expansion of the convolutional integral $\int_s^t S_{tu} \, dx^i_u \, f_i(y_u)$. To express this control with the highest accuracy, we are naturally led to consider the following semi-norms, that can be seen as adapted versions of the classical Hölder seminorms: if $y:[0,1] \to V$, $z:\cs_2 \to V$ and $h:\cs_3 \to V$, where $V$ is any Banach space and $\cs_3:=\{(t,u,s) \in [0,1]^3: \ s\leq u\leq t \}$, we denote, for any $\la >0$,
\begin{equation}
\cn[y;\cacha_1^\la([a,b];V)] :=\sup_{a\leq s<t \leq b} \frac{\norm{(\delha y)_{ts}}_V}{\lln t-s \rrn^\la} \quad , \quad \cn[y;\cac_1^0([a,b];V)]:=\sup_{t\in [a,b]} \norm{y_t}_V,
\end{equation}
\begin{equation}
\cn[z;\cac_2^\la([a,b];V)]:=\sup_{a\leq s <t\leq b} \frac{\norm{z_{ts}}_V}{\lln t-s \rrn^\la} \quad , \quad \cn[h;\cac_3^\la([a,b];V)]:=\sup_{a\leq s <u<t\leq b} \frac{\norm{h_{tus}}_V}{\lln t-s \rrn^\la}. 
\end{equation}
Then $\cacha_1^\la([a,b];V)$ stands for the set of paths $y:[0,1] \to V$ such that $\cn[y;\cacha_1^\la([a,b];V)]<\infty$, and we define $\cac_2^\la([a,b];V)$ and $\cac_3^\la([a,b];V)$ along the same lines. With this notation, observe for instance that if $y\in \cac_2^\la([a,b];\cl(V,W))$ and $z\in \cac_2^\be([a,b];V)$, the path $h$ defined as $h_{tus}=y_{tu}z_{us}$ ($s\leq u\leq t$) belongs to $\cac_3^{\la+\be}([a,b];W)$.

\smallskip

\noindent
When $[a,b]=[0,1]$, we will use the short form $\cac_k^\la(V):=\cac_k^\la([a,b];V)$.

\subsection{Infinite-dimensional rough path}\label{subsec:infinite-rough-path}

By anticipating the proof of Proposition \ref{prop:cas-regu}, we know that, when $x$ is a differentiable path, the expansion of $\int_s^t S_{tu} \, dx^i_u \, f_i(y_u)$ puts forward the three following operator-valued paths constructed from $x$:
$$\int_s^t S_{tu} \, dx^i_u  \quad ,  \quad \int_s^t a_{tu} \, dx^i_u \quad , \quad \int_s^t S_{tu} \, dx^i_u \, (\der x^j)_{us}.$$
A priori, these expressions do not make sense for a non-differentiable $\ga$-Hölder (rough)-path $x$. An integration by parts argument, retrospectively justified by Lemmas \ref{lem:cas-regu} and \ref{lem:cas-brown}, leads here to the general definition:

\begin{definition}\label{defi-chemin}
Under Assumptions (A1) and (X)$_{\ga}$, we define the three operator-valued paths $X^{x,i}$, $X^{ax,i}$ and $X^{xx,ij}$ by the formulas
\begin{equation}\label{defi-x-x}
X^{x,i}_{ts}:=S_{ts} (\der x^i)_{ts}-\int_s^t AS_{tu} (\der x^i)_{tu} \, du,
\end{equation}
\begin{equation}\label{defi-x-a-x}
X^{ax,i}_{ts}:=a_{ts} (\der x^i)_{ts}-\int_s^t AS_{tu} (\der x^i)_{tu} \, du,
\end{equation}
\begin{equation}\label{defi-x-x-x}
X^{xx,ij}_{ts}:=S_{ts} \xrgh^{\bt,ij}_{ts}-\int_s^t AS_{tu} \lc \xrgh^{\bt,ij}_{tu}+(\der x^i)_{tu}(\der x^j)_{us} \rc \, du.
\end{equation}
If in addition Assumption (F)$_1$ is satisfied, we set $F_{ij}(\vp):=f_i'(\vp) \cdot f_j(\vp)$ and we associate to every path $y:[0,1] \to \cb_p$ the two quantities
\begin{equation}\label{defi-j}
J^y_{ts}:=(\delha y)_{ts}-X^{x,i}_{ts}f_i(y_s)-X^{xx,ij}_{ts}F_{ij}(y_s),
\end{equation}
\begin{equation}\label{defi-k}
K^y_{ts}:=(\delha y)_{ts}-X^{x,i}_{ts}f_i(y_s).
\end{equation}
\end{definition}

\begin{lemma}\label{lem:cas-regu}
Suppose that $x$ is a $m$-dimensional differentiable path and let $\xrgh^\bt$ be its Lévy area, understood in the classical Lebesgue sense as the iterated integral $\xrgh^{\bt,ij}_{ts}:=\int_s^t dx^i_u \, (\der x^j)_{us}$. Then, under Assumption (A1),
\begin{equation}\label{ope-cas-regu}
X^{x,i}_{ts}=\int_s^t S_{tu} \, dx^i_u \quad , \quad X^{ax,i}_{ts}=\int_s^t a_{tu} \, dx^i_u \quad, \quad X^{xx,ij}_{ts}=\int_s^t S_{tu} \, dx^i_u \, (\der x^j)_{us}.
\end{equation}
\end{lemma}

\begin{proof}
As aforementioned, this is just a matter of integration by parts. For instance, one has
\bean
\int_s^t S_{tu} \, dx^i_u \, (\der x^j)_{us} &=& \int_s^t S_{tu} \, d_u \lp \xrgh^{\bt,ij}_{us}\rp\\
&=& \int_s^t S_{tu} \, d_u \lp -(\der \xrgh^{\bt,ij})_{tus}+\xrgh^{\bt,ij}_{ts}-\xrgh^{\bt,ij}_{tu} \rp\\
&=& \int_s^t S_{tu} \, d_u \lp -(\der x^i)_{tu}(\der x^j)_{us}-\xrgh^{\bt,ij}_{tu} \rp \\
&=&S_{ts} \xrgh^{\bt,ij}_{ts}-\int_s^t AS_{tu} \lc \xrgh^{\bt,ij}_{tu}+(\der x^i)_{tu}(\der x^j)_{us} \rc \, du.
\eean

\end{proof}

Observe now that the three expressions contained in (\ref{ope-cas-regu}) can also be directly interpreted as Itô integrals when $x$ stands for a standard Brownian motion. This interpretation remains consistent with Definition \ref{defi-chemin}:

\begin{lemma}\label{lem:cas-brown}
Suppose that $x$ is $m$-dimensional Brownian motion defined on a complete filtered probability space $(\Omega,\mathcal{F},P)$, and let $\xrgh^\bt$ be its Lévy area, understood in the Itô sense as the first iterated integral of $x$. Then, under Assumption (A1), the three identifications of the previous lemma remain valid in this context. 
\end{lemma}

\begin{proof}
It suffices to replace the integration by parts argument with Itô's formula, upon noticing that only Wiener integrals are involved here. For $X^{xx}$, we know indeed that for any fixed $s$, the process $u\mapsto \xrgh^{\bt,ij}_{us}=\int_s^u dx^i_v \, (\der x^j)_{vs}$ is a semimartingale and 
$$\int_s^t S_{tu} \, dx^i_u \, (\der x^j)_{us}=\int_s^t S_{tu} \, d_u(\xrgh^{\bt,ij}_{us}).$$
\end{proof}

To end up with this subsection, let us highlight the regularity properties that will be at our disposal throughout the study:

\begin{proposition}
Under Assumptions (A1) and (X)$_\ga$, one has, for all $\al \in (0,1),\ka \in [0,\ga)$,
\begin{equation}\label{regu-x-x}
X^{x,i} \in \cac_2^\ga(\cl(\cb_{\al,p},\cb_{\al,p})) \cap \cac_2^{\ga-\ka}(\cl(\cb_{\al,p},\cb_{\al+\ka,p})),
\end{equation}
\begin{equation}
X^{ax,i} \in \cac_2^{\ga+\al}(\cl(\cb_{\al,p},\cb_p)),
\end{equation}
\begin{equation}\label{regu-x-x-x}
X^{xx,ij} \in \cac_2^{2\ga}(\cl(\cb_{\al,p},\cb_{\al,p})) \cap \cac_2^{2\ga-\ka}(\cl(\cb_{\al,p},\cb_{\al+\ka,p})).
\end{equation}
We will denote by $\norm{X}_{\ga,\al,\ka}$ the norm attached to $X:=(X^x,X^{ax},X^{xx})$ through Properties (\ref{regu-x-x})-(\ref{regu-x-x-x}), that is to say
$$\norm{X}_{\ga,\al,\ka}:=\sum_{i,j=1}^m \lcl \cn[X^{x,i};\cac_2^\ga(\cl(\cb_{\al,p},\cb_{\al,p}))]+\ldots +\cn[X^{xx,ij};\cac_2^{2\ga-\ka}(\cl(\cb_{\al,p},\cb_{\al+\ka,p}))\rcl.$$
With this notation, one has $\norm{X}_{\ga,\al,\ka} \leq c_{\ga,\al,\ka} \norm{\xrgh}_\ga$. Moreover, if $\tilde{X}$ stands for the path associated to another trajectory $\tilde{x}$ satisfying (X)$_\ga$, then
\begin{equation}
\norm{X-\tilde{X}}_{\ga,\al,\ka} \leq c_{\ga,\al,\ka} \lcl 1+\norm{\xrgh}_\ga+\norm{\tilde{\xrgh}}_\ga \rcl \norm{\xrgh-\tilde{\xrgh}}_\ga.
\end{equation}

\end{proposition}

\begin{proof}
Properties (\ref{regu-x-x})-(\ref{regu-x-x-x}) are straightforward consequences of the well-known estimates (see \cite{pazy})
\begin{equation}\label{well-known-1}
\norm{S_{ts}\vp}_{\cb_{\al+\ka,p}} \leq c_\ka \lln t-s \rrn^{-\ka} \norm{\vp}_{\cb_{\al,p}} \quad , \quad \norm{AS_{ts} \vp}_{\cb_{\al+\ka,p}} \leq c_\ka \lln t-s \rrn^{-1-\ka} \norm{\vp}_{\cb_{\al,p}},
\end{equation}
\begin{equation}\label{well-known-2}
\norm{a_{ts}\vp}_{\cb_p} \leq c_\al \lln t-s \rrn^\al \norm{\vp}_{\cb_{\al,p}}.
\end{equation}
For example, for any $\vp \in \cb_{\al,p}$,
\bean
\norm{X^{x,i}_{ts}\vp}_{\cb_{\al+\ka,p}} &\leq & \norm{x}_\ga \lcl \lln t-s \rrn^\ga \norm{S_{ts} \vp}_{\cb_{\al+\ka,p}}+\int_s^t \lln t-u\rrn^\ga \norm{AS_{tu} \vp}_{\cb_{\al+\ka,p}} \, du \rcl\\
&\leq & c_{\ka}\norm{x}_\ga \norm{\vp}_{\cb_{\al,p}} \lcl \lln t-s \rrn^{\ga-\ka}+\int_s^t \lln t-u\rrn^{-1+\ga-\ka} \, du \rcl \\
& \leq &  c_{\ga,\ka}\norm{x}_\ga \norm{\vp}_{\cb_{\al,p}} \lln t-s \rrn^{\ga-\ka}. 
\eean
The controls of $\norm{X}_{\ga,\al,\ka}$ and $\norm{X-\tilde{X}}_{\ga,\al,\ka}$ can be readily checked from the very definitions (\ref{defi-x-x})-(\ref{defi-x-x-x}). Observe for instance that
\bean
\lefteqn{\| \int_s^t AS_{tu} (\der x^i)_{tu}(\der x^j)_{us} \, du-\int_s^t AS_{tu} (\der \tilde{x}^i)_{tu}(\der \tilde{x}^j)_{us} \, du\|_{\cl(\cb_p,\cb_p)}}\\
&\leq & \int_s^t \norm{AS_{tu}}_{\cl(\cb_p,\cb_p)}\lcl  \lln \der (x^i-\tilde{x}^i)_{tu} \rrn \lln (\der x^j)_{us}\rrn+\lln (\der \tilde{x}^i)_{tu}\rrn \lln \der (x^j-\tilde{x}^j)_{us} \rrn \rcl du\\
&\leq & c \lcl 1+\norm{\xrgh}_\ga+\norm{\tilde{\xrgh}}_\ga \rcl \norm{\xrgh-\tilde{\xrgh}}_\ga \lp \int_s^t \lln t-u\rrn^{-1+\ga} \lln u-s \rrn^\ga \, du\rp\\
&\leq & c \lln t-s \rrn^{2\ga} \lcl 1+\norm{\xrgh}_\ga+\norm{\tilde{\xrgh}}_\ga \rcl \norm{\xrgh-\tilde{\xrgh}}_\ga.
\eean

\end{proof}

\subsection{Interpretation of the equation}\label{subsec:notion-solu}

Let us now turn to the interpretation of (\ref{eq-base}) for a generic $2$-rough paths $\xrgh=(x,\xrgh^\bt)$. Like in \cite{Davie}, our approach is based on the Taylor expansion of the ordinary mild equation. We first give the general definition of a solution and then we clarify this definition by considering the two previously-known situations, namely when $x$ is a differentiable path and when $x$ is a standard Brownian motion. Remember that the notation $J^y$ has been introduced in Definition \ref{defi-chemin}.

\begin{definition}\label{defi-solu}
Under Assumptions (A1), (X)$_{\ga}$ and (F)$_1$, for all $\la \geq 0$ and $\psi \in \cb_{\la,p}$, we will call a solution in $\cb_{\la,p}$ of the equation
\begin{equation}\label{equa-gene}
y_t=S_t\psi+\int_0^t S_{t-u}f_i(y_u) \, dx^i_u \quad , \quad t\in [0,1],
\end{equation}
any path $y:[0,1] \to \cb_{\la,p}$ such that $y_0=\psi$ and there exists two coefficients $\mu >1, \ep >0$ for which
\begin{equation}\label{def-so-j-y}
J^y \in \cac_2^{\mu}([0,1];\cb_p) \quad \text{and} \quad J^y \in \cac_2^\ep([0,1];\cb_{\la,p}).
\end{equation}
\end{definition}

\begin{remark}
The reader familiar with the strategy of \cite{Davie} will not be surprised by the condition $J^y \in \cac_2^{\mu}([0,1];\cb_p)$ for some coefficient $\mu >1$. The second condition $J^y \in \cac_2^\ep([0,1];\cb_{\la,p})$ may be less expected. In fact, due to the property (\ref{well-known-2}), the fractional spaces $\cb_{\la,p}$ naturally arise from the controlled expansion of $\int_s^t S_{tu} \, dx^i_u \, f_i(y_u)$ (observe for instance (\ref{illustr})). 
\end{remark}

\begin{proposition}\label{prop:cas-regu}
Suppose that $x$ is a $m$-dimensional differentiable path, and let $\xrgh^\bt$ be its Lévy area, understood in the Lebesgue sense. We suppose that Assumptions (A1) and (F)$_1$ are both satisfied. Then, for all $\eta \in (0,1)$ and $\psi \in \cb_{\eta,p}$, the (ordinary) solution of Equation (\ref{equa-gene}) is also a solution in $\cb_{\eta,p}$ in the sense of Definition \ref{defi-solu}.
\end{proposition}

\begin{proof}
Let $y$ be the ordinary solution of (\ref{equa-gene}), with initial condition $\psi\in \cb_{\eta,p}$. Then $y\in \cac_1^0([0,1];\cb_{\eta,p})$ and since $(\delha y)_{ts}=\int_s^t S_{tu} \, dx^i_u \, f_i(y_u)$ and $f$ is bounded, one clearly has $y\in \cacha_1^1([0,1];\cb_p)$. Now, notice that owing to the identification (\ref{ope-cas-regu}), we get
$$K^y_{ts}=\int_s^t S_{tu} \, dx^i_u \, f_i(y_u)-X^{x,i}_{ts} f_i(y_s)=\int_s^t S_{tu} \, dx^i_u \, \der(f_i(y))_{us},$$
and so, due to (\ref{well-known-2}), one has
\begin{eqnarray}
\norm{K^y_{ts}}_{\cb_p} &\leq & \norm{\stackrel{.}{x}}_{\infty,[0,1]} \norm{f'}_\infty \int_s^t \norm{(\der y)_{us}}_{\cb_p} \, du \label{estim-direct-cas-regu}\\
& \leq & c_{x,f} \int_s^t \lcl \norm{(\delha y)_{us}}_{\cb_p}+\norm{a_{us}}_{\cl(\cb_{\eta,p},\cb_p)} \norm{y_s}_{\cb_{\eta,p}} \rcl \label{illustr}\\
&\leq & c_{x,f,y} \int_s^t \lcl \lln u-s\rrn+\lln u-s\rrn^\eta \rcl \, du \ \leq \ c_{x,f,y} \lln t-s \rrn^{1+\eta}.\nonumber
\end{eqnarray}
To complete the proof, observe that by resorting to the identification (\ref{ope-cas-regu}) once again, we can write $J^y_{ts}=\int_s^t S_{tu}\, dx^i_u \, M^i_{us}$, with
\bean
M^i_{us}&=& \der (f_i(y))_{us}-(\der x^j)_{us} f_i'(y_s) \cdot f_j(y_s)\\
&=& \int_0^1 dr \,  f_i'(y_s+r(\der y)_{us}) \cdot (\der y)_{us}-(\der x^j)_{us} f_i'(y_s) \cdot f_j(y_s)\\
&=& \int_0^1 dr \,  f_i'(y_s+r(\der y)_{us}) \cdot a_{us}y_s\\
& &+\int_0^1 dr \,  f_i'(y_s+r(\der y)_{us}) \cdot (\delha y)_{us}-(\der x^j)_{us} f_i'(y_s) \cdot f_j(y_s),
\eean
and thus
\begin{multline}\label{exp-m-i}
M^i_{us}=\int_0^1 dr \, f_i'(y_s+r(\der y)_{us}) \cdot a_{us} y_s+\int_0^1 dr \, f_i'(y_s+r(\der y)_{us}) \cdot K^y_{us}\\
+\int_0^1 dr \, f_i'(y_s+r(\der y)_{us}) \cdot X^{ax,j}_{us}f_j(y_s)+\int_0^1 dr \, \lc f_i'(y_s+r(\der y)_{us})-f_i'(y_s)\rc \cdot (\der x^j)_{us}f_j(y_s),
\end{multline}
where we have used the trivial relation $X^{x,i}_{us}=X^{ax,i}_{us}+(\der x^j)_{us}$. From this expression, it is easy to show that $\norm{M^i_{us}}_{\cb_p} \leq c_y \lln u-s \rrn^\eta$, which leads to (\ref{def-so-j-y}) with $\mu=1+\eta$, $\ep=1$.
\end{proof}

\begin{proposition}\label{prop:cas-brown}
Suppose that $x$ is a $m$-dimensional standard Brownian motion defined on a complete filtered probability space $(\Omega,\mathcal{F},P)$, and let $\xrgh^\bt$ be its Lévy area, understood in the Itô sense. Suppose also that Assumptions (A1) and (F)$_2$ are both satisfied. Then, for all $\eta \in (1/2,1)$ and $\psi\in \cb_{\eta,p}$, the Itô solution of Equation (\ref{equa-gene}) is almost surely a solution in $\cb_{\eta,p}$ in the sense of Definition \ref{defi-solu}.
\end{proposition}

\begin{proof}
For the sake of clarity, we have postponed the proof of this result to Appendix B.
\end{proof}

Together with the forthcoming uniqueness result contained in Theorem \ref{theo-uni}, the above-stated properties allow to identify, in the two reference situations (i.e., when $x$ is a differentiable path and when $x$ is a standard Brownian motion), the solution in the sense of Definition \ref{defi-solu} with the classical solution. We will then lean on the continuity Theorem \ref{theo-conti} to fully justify our interpretation of (\ref{equa-gene}) (see Remark \ref{rk:discuss-sol}). 

\subsection{Main results}\label{subsec:main-results}

With the tools and the definitions we have just introduced, we are in a position to state the three main results of this paper, which successively provide the existence, uniqueness and continuity of the solution to (\ref{equa-gene}).

\begin{theorem}\label{theo-exi}
Under Assumptions (A1), (X)$_\ga$ and (F)$_2$, for all $\ga'\in (1-\ga,\ga+1/2)$ and $\psi\in \cb_{\ga',p}$, Equation (\ref{equa-gene}) admits a solution $y$ in $\cb_{\ga',p}$ in the sense of Definition \ref{defi-solu}, which satisfies 
$$\cn[y;\cacha_1^\ga([0,1];\cb_p)]+\cn[y;\cac_1^0([0,1];\cb_{\ga',p})] \leq C(\norm{\xrgh}_\ga,\norm{\psi}_{\cb_{\ga',p}}),$$
for some function $C:(\R^+)^2 \to \R$ growing with its arguments.
\end{theorem}

\begin{theorem}\label{theo-uni}
If $p>n$ and if Assumptions (A1), (A2), (X)$_\ga$ and (F)$_3$ are all satisfied, then for all $\ga'\in (1-\ga,\ga+1/2)$ and $\psi \in \cb_{\ga',p}$, the solution $y$ in $\cb_{\ga',p}$ given by Theorem \ref{theo-exi} is unique. Moreover, for any
$$0<\be < \inf\lp 3\ga-1,\ga+\ga'-1,\ga-(\ga'-1/2)\rp,$$
there exists a constant $c_{x,\psi,f,\be}$ such that for all $n$,
$$\max_{k=0,\ldots ,2^n} \norm{y_{t_k^n}-y^n_{t_k^n}}_{\cb_{\ga',p}} \leq \frac{c_{x,\psi,f,\be}}{(2^n)^\be},$$
where $y^n$ stands for the path given by the discrete scheme (\ref{schema}).
\end{theorem}

\begin{theorem}\label{theo-conti}
Under the assumptions of Theorem \ref{theo-uni}, the solution of (\ref{equa-gene}) is continuous with respect to the initial condition and the driving rough path. More precisely, if $y$ (resp. $\tilde{y}$) is the solution in $\cb_{\ga',p}$ associated to $(x,\xrgh^\bt)$ (resp. $(\tilde{x},\tilde{\xrgh}^\bt)$), with initial condition $\psi$ (resp. $\tilde{\psi}$), then
\begin{multline}\label{ito}
\cn[y-\yti;\cacha_1^\ga([0,1];\cb_p)]+\cn[y-\yti;\cac_1^0([0,1];\cb_{\ga',p})] \\
\leq C\lp \norm{\xrgh}_\ga,\norm{\tilde{\xrgh}}_\ga,\norm{\psi}_{\cb_{\ga',p}},\norm{\tilde{\psi}}_{\cb_{\ga',p}} \rp \lcl \norm{\psi-\tilde{\psi}}_{\cb_{\ga',p}}+\norm{\xrgh-\tilde{\xrgh}}_\ga \rcl,
\end{multline}
for some functions $C:(\R^+)^4 \to \R^+$ growing with its arguments.
\end{theorem}

\smallskip

Together with the identification result established in Proposition \ref{prop:cas-regu}, these three theorems offer another perspective on the solution of Equation (\ref{equa-gene}), which may be more in accordance with the formalism used in \cite{FVbook} for rough standard systems:

\begin{corollary}\label{coro:sol-rough-path}
Under the assumptions of Theorem \ref{theo-uni}, suppose that $\psi \in \cb_{\ga',p}$ and let $(\xti^n)_n$ be a sequence of differentiable paths such that $\norm{x-\xti^n}_\ga+\norm{\xrgh^\bt-\tilde{\xrgh}^{\bt,n}}_{2\ga} \to 0$ as $n$ tends to infinity, where $\tilde{\xrgh}^{\bt,n}$ stands for the standard Lévy area constructed from $\xti^n$. Let $\yti^n$ be the (ordinary) solution of (\ref{equa-gene}) associated to each $\xti^n$. If $y$ is the solution of (\ref{equa-gene}) given by Theorem \ref{theo-uni}, then
\begin{equation}\label{lip-ito}
\cn[y-\yti^n;\cacha_1^\ga([0,1];\cb_p)]+\cn[y-\yti^n;\cac_1^0([0,1];\cb_{\ga',p})] \to 0
\end{equation}
as $n$ tends to infinity.
\end{corollary}

\

\begin{remark}\label{rk:discuss-sol}
Through the latter result, one can see that the exhibited solution $y$ is a solution \emph{in the rough paths sense}, that is to say a limit of ordinary solutions with respect to some particular topology (compare with \cite[Definition 10.17]{FVbook}). In this context, $y$ can legitimately be called a mild solution of (\ref{eq-intro}), as a limit of classical mild solutions. Furthermore, it is worth noticing that given the regularity assumptions on $f_i$, if we suppose in addition that the initial condition $\psi$ belongs to the domain $\cd(A_p)$, then each (ordinary) mild solution $\yti^n$ is also a strong solution (see \cite[Theorem 6.1.6]{pazy}). Consequently, if $\psi \in \cd(A_p)$, $y$ can also be considered as a strong solution of (\ref{eq-intro}), keeping in mind the topology of the underlying convergence result (\ref{lip-ito}).
\end{remark}

\subsection{Extension to rougher paths}\label{subsec:rougher-path}

Before we turn to the proof of Theorems \ref{theo-exi}-\ref{theo-conti}, let us say a few words about the possibility of extending these results to a rougher path $x$, or otherwise stated when the Hölder coefficient $\ga$ is smaller than $1/3$.

\smallskip

Remember that for standard finite-dimensional rough systems, the results obtained by Davie in \cite{Davie} have been generalized to any $\ga \in (0,1)$ by Friz and Victoir (\cite{FV4}): essentially, the system (\ref{eq-std}) can be interpreted and solved provided that (i) the vector field $\si_{ij}$ is smooth enough and (ii) one is able to construct the iterated integrals of $x$ up to the $k$-th order, where $\frac{1}{k+1} < \ga \leq \frac{1}{k}$.

\smallskip

As far as Equation (\ref{equa-gene}) is concerned, let us first consider the next step of the procedure, which corresponds to $\frac{1}{4} < \ga \leq \frac{1}{3}$. For more simplicity, we assume that $f_i$ is infinitely differentiable with bounded derivatives. Suppose for the moment that $x$ is a differentiable path, and let us introduce, on top of $(X^x,X^{ax},X^{xx})$, the two additional operator-valued paths
$$X^{axx}_{ts}=\int_s^t a_{tu} \, dx^i_u \, (\der x^j)_{us} \quad , \quad X^{xxx,ijk}_{ts}=\int_s^t S_{tu} \, dx^i_u \, \xrgh^{\bt,jk}_{us}.$$
Let us also define $F^1_i(\vp)=f_i(\vp)$, $F^2_{ij}(\vp)=f_i'(\vp) \cdot f_j(\vp)$, $F^3_{ijk}(\vp)=f_i''(\vp) \cdot f_j(\vp) \cdot f_k(\vp)+f_i'(\vp) \cdot f_j'(\vp) \cdot f_k(\vp)$, and the three intermediate quantities
$$L^y_{ts}=(\delha y)_{ts}-X^{x,i}_{ts} F^1_i(y_s) \quad , \quad K^y_{ts}=(\delha y)_{ts}-X^{x,i}_{ts} F^1_i(y_s)-X^{xx,ij}_{ts} F^2_{ij}(y_s),$$
$$J^y_{ts}=(\delha y)_{ts}-X^{x,i}_{ts} F^1_i(y_s)-X^{xx,ij}_{ts} F^2_{ij}(y_s)-X^{xxx,ijk}_{ts} F^3_{ijk}(y_s).$$
Once endowed with this notation, a Taylor expansion of the (ordinary) equation (\ref{equa-gene}), similar to (\ref{exp-m-i}), shows that for all $s<t\in [0,1]$, one has 
\begin{equation}\label{decompo-1/4}
(\delha y)_{ts}=X^{x,i}_{ts} F^1_i(y_s)+X^{xx,ij}_{ts} F^2_{ij}(y_s)+X^{xxx,ijk}_{ts} F^3_{ijk}(y_s)+\int_s^t S_{tu} \, dx^i_u \, y^{\sharp,i}_{us},
\end{equation}
where the 'residual' path $y^\sharp$ can be decomposed as $y^{\sharp,i}=y^{\sharp,i,1}+y^{\sharp,i,2}$, with
\begin{eqnarray}\label{y-sharp-1}
y^{\sharp,i,1}_{us} &=& \int_0^1 dr \, f_i'(y_s+r (\der y)_{us}) \cdot \lc a_{us}y_s+X^{ax,j}_{us} F^1_j(y_s)+X^{axx,jk}_{us} F^2_{jk}(y_s) \rc \nonumber\\
& & +\int_0^1 dr \int_0^1 dr' \, r f_i''(y_s+r(\der y)_{us}) \cdot (\der x^j)_{us} f_j(y_s) \cdot X^{ax,k}_{us}F^1_k(y_s),
\end{eqnarray}
\begin{eqnarray}\label{y-sharp-2}
\lefteqn{y^{\sharp,i,2}_{us}=}\nonumber\\
& & \int_0^1 dr \, f_i'(y_s+r(\der y)_{us}) \cdot K^y_{us} \nonumber\\
& &+\int_0^1 dr \lc f_i'(y_s+r(\der y)_{us})-f_i'(y_s) \rc \cdot \xrgh^{\bt,jk}_{us}\cdot f_i'(y_s)\cdot f_j'(y_s)\cdot f_k(y_s) \nonumber\\
& &+\int_0^1 dr \int_0^1 dr' \, r f_i''(y_s+r(\der y)_{us}) \cdot L^y_{us} \cdot (\der x^j)_{us} \cdot f_j(y_s) \nonumber\\
& &+\int_0^1 dr \int_0^1 dr' \, r \lc f_i''(y_s+r(\der y)_{us})-f_i''(y_s) \rc \cdot (\der x^j)_{us} (\der x^k)_{us} \cdot f_j(y_s) \cdot f_k(y_s).
\end{eqnarray}
By looking closely at these expressions, it is not difficult to realize that the arguments displayed in the forthcoming sections \ref{sec:existence}-\ref{sec:conti} can be adapted to the decomposition (\ref{decompo-1/4}) so as to handle the case where $\ga \in (\frac{1}{4},\frac{1}{3}]$ (compare for instance (\ref{y-sharp-1})-(\ref{y-sharp-2}) with (\ref{compare-1})-(\ref{compare-4})). This supposes that the intermediate paths $J^y,K^y,L^y$ should be controlled with the respective topologies $\cac_2^{4\ga}(\cb_p) \cap \cac_2^\ep(\cb_{\ga',p}), \cac_2^{3\ga}(\cb_p),\cac_2^{2\ga}(\cb_p)$, and that the space parameter $\ga'$ should be picked in the (non-empty) interval $(1-\ga,\ga+1/2)$, as in Theorems \ref{theo-exi}-\ref{theo-conti}. This also supposes, in order to extend the path $X^{xxx}$, that $x$ allows the construction of a $3$-rough path $\xrgh=(x,\xrgh^\bt=\iint dxdx,\xrgh^{\textbf{3}}=\iiint dx dx dx) \in \cac_1^\ga \times \cac_2^{2\ga} \times \cac_2^{3\ga}$. We know that this assumption covers in particular the case of a fractional Brownian motion with Hurst index $H>1/4$, see \cite{CQ}.

\smallskip

The situation gets more complicated as soon as $\ga < 1/4$, since we can no longer pick $\ga'$ in the (now empty) interval $(1-\ga,\ga+1/2)$, and this assumption played a fundamental role in our estimates. Indeed, on the one hand, the condition $\ga'>1-\ga$ ensures that the order of the terms derived from (\ref{y-sharp-1}) or (\ref{compare-2}) is greater than $\ga+\ga'>1$, or otherwise stated that these paths can be considered as residual terms. On the other hand, the condition $\ga'< \ga +1/2$ is used in some estimates like (\ref{line-cont}) to go from $\cb_{\ga',p}$ to $\cb_{1/2,p}$ and thus take profit of the linear control (\ref{contr-nem-1}) (instead of the quadratic control (\ref{contr-nem-2})). Therefore, when $\ga < 1/4$, some sharpness is to be lost in our estimates and the method under consideration in this paper would only provide us with a local solution, on a time interval linked to $x$, $f$ and $\psi$. To overcome this difficulty, it may be useful to modify the path $(X^x,X^{ax},X^{xx},X^{axx},X^{xxx},\ldots)$ into a more appropriate trajectory, which would for instance includes mixed operators such as
\begin{equation}\label{mixed-operator}
X^{xa,i}_{ts}(\vp_1,\vp_2)=\int_s^t S_{tu} \, dx^i_u \lc a_{us} \vp_1 \cdot \vp_2 \rc \quad , \quad \vp_1,\vp_2\in \cb.
\end{equation}
Observe however that the extension of (\ref{mixed-operator}) to a generic $\ga$-Hölder path $x$ (with $\ga<1$) can no longer be done via an integration by parts argument (as in Lemma \ref{lem:cas-regu}), which considerably increases the difficulty of the study.

\section{Existence of a solution}\label{sec:existence}

This section is devoted to the proof of Theorem \ref{theo-exi}. Thus, we henceforth suppose that the assumptions of the theorem, namely (A1), (X)$_\ga$ and (F)$_2$, are all satisfied. Besides, we fix a parameter $\ga'\in (1-\ga,\ga+1/2)$ and an initial condition $\psi \in \cb_{\ga',p}$.

\subsection{Additional notation}\label{subsec:notations}

We consider the sequence $(\Pi^n)_n$ of dyadic partitions of $[0,1]$ (i.e., $t_k^n=\frac{k}{2^n}$) and we define the discrete path $y^n$ following the iteration formula: 
\begin{equation}\label{schema}
y^n_0:=\psi \quad , \quad y^n_{t_{k+1}^n}:=S_{t_{k+1}^nt_k^n}y^n_{t_k^n}+X^{x,i}_{t_{k+1}^n t_k^n} f_i(y^n_{t_k^n})+X^{xx,ij}_{t_{k+1}^nt_k^n} F_{ij}(y^n_{t_k^n}) \quad , \quad t_k^n \in \Pi^n.
\end{equation}
The path $y^n$ is then extended on $[0,1]$ by linear interpolation. For the sake of clarity, we will denote in this section $J^n:=J^{y^n}$ and $K^n:=K^{y^n}$. Observe that owing to the very definition of $y^n$, one has $J^n_{t_{k+1}^n t_k^n}=0$.

\smallskip

\noindent
In the rest of the paper, we will also appeal to the discrete versions of the generalized Hölder norms introduced in Subsection \ref{subsec:incre}. Thus, for any $n\in \N$, we denote $\llbracket a,b \rrbracket_n:=[a,b] \cap \Pi^n$ and 
$$\cn[h;\cacha_1^\la(\llbracket t_p^n,t_q^n \rrbracket_n,\cb_{\al,p})] :=\sup_{\substack{t_p^n \leq s <t\leq t_q^n\\ s,t\in \Pi^n}} \frac{\norm{(\delha h)_{ts}}_{\cb_{\al,p}}}{\lln t-s\rrn^\la},$$
We define the two quantities $\cn[.;\cac_2^\la(\llbracket a,b \rrbracket_n;\cb_{\al,p})]$ and $\cn[.;\cac_3^\la(\llbracket a,b \rrbracket_n;\cb_{\al,p})]$ along the same lines.

\smallskip

\noindent
Now, for any (not necessarily uniform) partition $\tilde{\Pi}$ of $[0,1]$ made of points of $\Pi^n$, we define the path $J^{n,\tilde{\Pi}}$ for all $s<t\in \Pi^n$ as
$$J^{n,\tilde{\Pi}}_{ts}:=\begin{cases}
0 & \text{if} \ (s,t) \cap \tilde{\Pi}=\emptyset\\
(\delha J^n)_{t us} & \text{if} \ (s,t) \cap \tilde{\Pi} =u\\
J^n_{ts}-J^n_{t\tilde{t}_N}-\sum_{k=1}^{N-1}S_{t\tilde{t}_{k+1}} J^n_{\tilde{t}_{k+1}\tilde{t}_k}-S_{t\tilde{t}_1}J^n_{\tilde{t}_1 s} & \text{if} \ (s,t) \cap \tilde{\Pi}=\{\tilde{t}_1,...,\tilde{t}_N \}
\end{cases}$$

\smallskip

\begin{remark}\label{rk:partitions}
Since $J^n_{t_{k+1}^n t_k^n}=0$, one has in particular $J^{n,\Pi^n}_{ts}=J^n_{ts}$. Besides, if $\tilde{\Pi},\hat{\Pi}$ are two partitions of $[0,1]$ made of points of $\Pi^n$ and such that $\tilde{\Pi} \cap (s,t)=\{\tilde{t}_1,\ldots,\tilde{t}_N \}$ ($N\geq 3$) and $\hat{\Pi}\cap (s,t)=\{\tilde{t}_1,\ldots,\tilde{t}_{k-1},\tilde{t}_{k+1},\ldots \tilde{t}_N \}$ for $1\leq k\leq N-1$, then $J^{n,\tilde{\Pi}}_{ts}-J^{n,\hat{\Pi}}_{ts}=S_{t\tilde{t}_{k+1}} (\delha J^n)_{\tilde{t}_{k+1}\tilde{t}_k\tilde{t}_{k-1}}$. With the same notation, if $\hat{\Pi}\cap (s,t)=\{\tilde{t}_1,\ldots,\tilde{t}_{N-1}\}$, then $J^{n,\tilde{\Pi}}_{ts}-J^{n,\hat{\Pi}}_{ts}=(\delha J^n)_{t\tilde{t}_N\tilde{t}_{N-1}}$.
\end{remark}

\

\subsection{Preliminary results on $J^n$}

We fix $t_p^n < t_q^n \in \Pi^n$ and we apply the algorithm described in Appendix A to the uniform partition $\{t_p^n,t_{p+1}^n,\ldots,t_q^n\}$. Set $N:=q-p$, and so, for any $k\in \{ p,\ldots,q\}$, $t_k^n=t_p^n+\frac{(k-p)(t_q^n-t_p^n)}{N}$. We also denote by $(\Pi^{n,m})_{m\in \{1,\ldots,N-1\}}$ the (decreasing) sequence of partitions of $[t_p^n,t_q^n]$ deduced from the algorithm, and $\Pi^{n,0}:=\{t_p^n,t_{p+1}^n,\ldots,t_q^n\}$. Finally, set $J^{n,m}_{t_q^nt_p^n}:=J^{n,\Pi^{n,m}}_{t_q^n t_p^n}$. With this notation in hand, one has
$$J^n_{t_q^nt_p^n}=\sum_{m=0}^{N-1} \lcl J^{n,m}_{t_q^n t_p^n}-J^{n,m+1}_{t_q^n t_p^n}\rcl.$$
Once endowed with this decomposition, we can show the following result, which turns out to be the starting point of our reasoning:

\begin{lemma}\label{lemme:base}
For all $\mu >1$ and $\ka>0$, there exists a constant $c=c_{\mu,\ka}$ such that
\begin{multline}\label{contr-1}
\norm{J^n_{t_q^n t_p^n}}_{\cb_{\ga',p}} \leq c \lcl \lln t_q^n-t_p^n \rrn^\ka+\lln t_q^n-t_p^n \rrn^{\mu-\ga'}\rcl \\
\lcl \cn[\delha J^n;\cac_3^\ka(\llbracket t_p^n,t_q^n\rrbracket_n;\cb_{\ga',p})]+\cn[\delha J^n;\cac_3^\mu(\llbracket t_p^n,t_q^n\rrbracket_n;\cb_p)] \rcl,
\end{multline}
and
\begin{equation}\label{contr-2}
\norm{J^n_{t_q^nt_p^n}}_{\cb_p} \leq c \lln t_q^n-t_p^n\rrn^\mu \cn[\delha J^n;\cac_3^\mu(\llbracket t_p^n,t_q^n\rrbracket_n;\cb_p)].
\end{equation}
 
\end{lemma}

\begin{proof}
We use the notation of Appendix A. By refering to Remark \ref{rk:partitions}, one easily deduces
\begin{multline*}
\sum_{m=0}^{N-1} \lcl J^{n,m}_{t_q^n t_p^n}-J^{n,m+1}_{t_q^n t_p^n} \rcl\\
=\sum_{r=1}^{M-1} \lcl (\delha J^n)_{t_q^n t^n_{p+k_{A_{r-1}+1}}t^n_{p+k_{A_{r-1}+1}^-}}+\sum_{m=A_{r-1}+2}^{A_{r}} S_{t_q^n t^n_{p+k_m^+}} (\delha J^n)_{t^n_{p+k_m^+}t^n_{p+k_m}t^n_{p+k_m^-}} \rcl\\
+(\delha J^n)_{t_q^n t^n_{p+k_{A_{M-1}+1}} t^n_{p+k^-_{A_{M-1}+1}}}+\mathbf{1}_{\{A_{M-1}+1 < N-1\}} (\delha J^n)_{t_q^n t^n_{p+k_{N-1}}t^n_p}.
\end{multline*}
Then, if $C_n:= \cn[\delha J^n;\cac_3^\ka(\llbracket t_p^n,t_q^n\rrbracket_n;\cb_{\ga',p})]+\cn[\delha J^n;\cac_3^\mu(\llbracket t_p^n,t_q^n\rrbracket_n;\cb_p)]$, one has
\bean
\lefteqn{\sum_{m=0}^{N-1} \norm{J^{n,m}_{t_q^nt_p^n}-J^{n,m+1}_{t_q^n t_p^n}}_{\cb_{\ga',p}}}\\
&\leq & 2 \, C_n \lln t_q^n-t_p^n \rrn^{\ka}\\
& &+C_n\sum_{r=0}^{M-1}\lcl \lln t_q^n-t^n_{p+k^-_{A_{r-1}+1}}\rrn^{\ka}+\sum_{m=A_{r-1}+2}^{A_{r}} \lln t_q^n-t^n_{p+k_m^+} \rrn^{-\ga'} \lln t^n_{p+k_m^+}-t^n_{p+k_m^-} \rrn^\mu \rcl \\
&\leq & C_n \lcl \lln t_q^n-t_p^n \rrn^{\ka}+\lln t_q^n-t_p^n \rrn^{\mu-\ga'}\rcl\\
& & \lcl 2+\sum_{r=0}^{M-1} \lcl \lln 1-\frac{k^-_{A_{r-1}+1}}{N} \rrn^{\ka}+\frac{1}{N^\mu} \sum_{m=A_{r-1}+2}^{A_{r}} \lln 1-\frac{k_m^+}{N}\rrn^{-\ga'} \lln k_m^+-k_m^-\rrn^\mu \rcl\rcl\\
&\leq & c_{\ka,\mu,\ga'}  \lcl \lln t_q^n-t_p^n \rrn^{\ka}+\lln t_q^n-t_p^n \rrn^{\mu-\ga'}\rcl C_n, 
\eean
thanks to Proposition \ref{resu-algo}. The second control (\ref{contr-2}) can be obtained with the same arguments, upon noticing that (\ref{contr-unif}) entails in particular 
$$\sum_{r=1}^{M-1} \lcl \lln 1-\frac{k_{A_{r-1}+1}^-}{N} \rrn^{\mu}+\frac{1}{N^\mu}\sum_{m=A_{r-1}+2}^{A_{r}}  \lln k_m^+-k_m^- \rrn^\mu \rcl \leq c_{\mu} < \infty.$$

\end{proof}

\begin{lemma}
For every path $y:[0,1] \to \cb_p$ and all $s<u<t \in [0,1]$,
\begin{equation}\label{decompo-delha-j-n-2}
(\delha J^y)_{tus}=X^{x,i}_{tu} \der(f_i(y))_{us}-X^{x,i}_{tu} (\der x^j)_{us} F_{ij}(y_s)+X^{xx,ij}_{tu} \der(F_{ij}(y))_{us} 
\end{equation}
and also
\begin{equation}\label{dec-delha-j-n}
(\delha J^y)_{tus}=I_{tus}+II_{tus}+III_{tus}+IV_{tus},
\end{equation}
with
\begin{equation}\label{compare-1}
I_{tus}:=X^{x,i}_{tu} \lp \int_s^1 dr \, f_i'(y_s+r(\der y)_{us}) \cdot K^y_{us} \rp,
\end{equation}
\begin{equation}\label{compare-2}
II_{tus}:=X^{x,i}_{tu} \lp \int_s^1 dr \, f_i'(y_s+r(\der y)_{us}) \cdot \lcl a_{us}y_s+X^{ax,j}_{us}f_j(y_s) \rcl \rp,
\end{equation}
\begin{equation}\label{compare-3}
III_{tus}:=X^{x,i}_{tu} \lp \int_0^1 dr \, \lc f_i'(y_s+r(\der y)_{us})-f_i'(y_s)\rc \cdot (\der x^j)_{us} f_j(y_s) \rp,
\end{equation}
\begin{equation}\label{compare-4}
IV_{tus}:=X^{xx,ij}_{tus} \der (F_{ij}(y))_{us}.
\end{equation}
\end{lemma}

\begin{proof}
Those are only straightforward expansions. For (\ref{decompo-delha-j-n-2}), we use the fact that if $m_{ts}:=g_{ts} h_s$, then $(\delha m)_{tus}=(\delha g)_{tus}h_s-g_{tu} (\der h)_{us}$, together with the algebraic relations
$$(\delha X^{x,i})_{tus}=0 \quad , \quad (\delha X^{xx,ij})_{tus}=X^{x,i}_{tu} (\der x^j)_{us}  \quad \quad \text{for all} \ s\leq u\leq t,$$
that can be readily checked from the expressions (\ref{defi-x-x}) and (\ref{defi-x-x-x}). The expansion of $\der(f_i(y))_{us}-(\der x^j)_{us} F_{ij}(y_s)$ which then leads to (\ref{dec-delha-j-n}) has already been elaborated on in the proof of Proposition \ref{prop:cas-regu}, see (\ref{exp-m-i}).
\end{proof}

\subsection{Existence of a solution}

Thanks to the above preliminary results, we are first able to control $J^n$ on successive time intervals independent of $n$:

\begin{proposition}\label{prop:inter-exi}
Suppose that $\mu,\ep$ satisfy
\begin{equation}\label{jeu-coeff-2}
3\ga > \mu >1 \quad , \quad \ga+\ga'>\mu >1 \quad , \quad \ga-(\ga'-\frac{1}{2}) >\ep > 0.
\end{equation}
Then there exists a time $T_0=T_0(x,f,\ga,\ga',\mu, \ep) >0$, $T_0\in \Pi^n$, such that for any $k$,
\begin{equation}\label{contr-sob-j}
\cn[J^n;\cac_2^\ep(\llbracket kT_0,(k+1)T_0 \wedge 1 \rrbracket_n;\cb_{\ga',p})] \leq 1+\norm{y^n_{kT_0}}_{\cb_{\ga',p}}
\end{equation}
and
\begin{equation}\label{contr-b-p-2}
\cn[J^n;\cac_2^\mu(\llbracket kT_0,(k+1)T_0 \wedge 1 \rrbracket_n;\cb_{p})] \leq 1+\norm{y^n_{kT_0}}_{\cb_{\ga',p}}.
\end{equation}
\end{proposition}

\begin{proof}
This is an iteration procedure over the points of the partition, for which we first focus on the case $k=0$ in (\ref{contr-sob-j}) and (\ref{contr-b-p-2}). Assume that both estimates hold true on $\llbracket 0,t_q^n \rrbracket_n$. Then, for any $t\in \llbracket 0,t_q^n \rrbracket_n$, one has, thanks to (\ref{regu-x-x}), (\ref{regu-x-x-x}) and (\ref{contr-nem-1}),
\begin{eqnarray}
\norm{y_t^n}_{\cb_{\ga',p}} &\leq & \norm{J^n_{t0}}_{\cb_{\ga',p}}+\norm{S_{t0}\psi}_{\cb_{\ga',p}}+c^0_{x}t^{\ga-(\ga'-\frac{1}{2})} \lcl \norm{f_i(\psi)}_{\cb_{1/2,p}}+\norm{F_{ij}(\psi)}_{\cb_{1/2,p}}\rcl \label{line-cont}\\
&\leq & \norm{J^n_{t0}}_{\cb_{\ga',p}}+\norm{S_{t0}\psi}_{\cb_{\ga',p}}+c^1_{x,f}t^{\ga-(\ga'-\frac{1}{2})} \lcl 1+\norm{\psi}_{\cb_{\ga',p}}\rcl \nonumber\\
&\leq & c^2_{x,f} \lcl 1+\norm{\psi}_{\cb_{\ga',p}} \rcl,\label{contr-norm-sob}
\end{eqnarray}
and so $\cn[y^n;\cac_1^0(\llbracket 0,t_q^n \rrbracket_n,\cb_{\ga',p})] \leq c^2_{x,f} \lcl 1+\norm{\psi}_{\cb_{\ga',p}} \rcl$. Besides, if $s<t \in \llbracket 0,t_q^n \rrbracket_n$,
\bean
\norm{(\delha y^n)_{ts}}_{\cb_{p}} &\leq & \norm{J^n_{ts}}_{\cb_{p}}+\norm{X^{x,i}_{ts}f_i(y^n_s) }_{\cb_{p}}+\norm{X^{xx,ij}_{ts}F_{ij}(y^n_s)}_{\cb_{p}}\\
&\leq & \lln t-s \rrn^\ga c^3_{x,f} \lcl 1+\norm{\psi}_{\cb_{\ga',p}} \rcl,
\eean
and hence
\begin{equation}\label{est-2}
\cn[y^n;\cacha_1^\ga(\llbracket 0,t_q^n \rrbracket_n;\cb_{p})] \leq c^3_{x,f} \lcl 1+\norm{\psi}_{\cb_{\ga',p}} \rcl.
\end{equation}
One can also rely on the estimate
\begin{equation}\label{est-3}
\norm{K^n_{ts}}_{\cb_p} \leq \norm{J^n_{ts}}_{\cb_p}+\norm{X^{xx,ij}_{ts}F_{ij}(y^n_s)}_{\cb_p}\leq c^4_{x,f} \lln t-s \rrn^{2\ga} \lcl 1+\norm{\psi}_{\cb_{\ga',p}} \rcl.
\end{equation}
Now, from the decomposition (\ref{dec-delha-j-n}), we easily deduce, for all $0\leq s < u<t \in \llbracket 0,t^n_{q+1}\rrbracket_n$,
$$
\norm{(\delha J^n)_{tus}}_{\cb_p}
\leq c^5_{x,f} \lcl 1+\norm{\psi}_{\cb_{\ga',p}}\rcl \lcl \lln t-s \rrn^{3\ga}+\lln t-s\rrn^{\ga+\ga'} \rcl.
$$
Indeed, one has for instance
\bean
\lefteqn{\norm{\lc f_i'(y_s+r(\der y)_{us})-f_i'(y_s)\rc \cdot (\der x^j)_{us}f_j(y_s)}_{\cb_p}}\\
&\leq & c_{x,f} \lln u-s \rrn^\ga \norm{(\der y)_{us}}_{\cb_p}\\
&\leq & c_{x,f} \lln u-s \rrn^\ga \lcl \norm{(\delha y)_{us}}_{\cb_p}+\norm{a_{us}y_s}_{\cb_p} \rcl\\
&\leq & c_{x,f} \lcl 1+\norm{\psi}_{\cb_{\ga',p}} \rcl \lcl \lln u-s \rrn^{2\ga}+\lln u-s \rrn^{\ga+\ga'} \rcl \ \leq \ c_{x,f}\lln u-s \rrn^{2\ga} \lcl 1+\norm{\psi}_{\cb_{\ga',p}} \rcl,
\eean
where the constant $c_{x,f}$ may of course vary from line to line. Consequently,
$$\cn[\delha J^n;\cac_3^\mu(\llbracket 0,t_{q+1}^n \rrbracket_n ;\cb_{p})] \leq c^5_{x,f} \lcl 1+\norm{\psi}_{\cb_{\ga',p}} \rcl \lcl T_0^{3\ga-\mu}+T_0^{\ga+\ga'-\mu} \rcl.$$
On the other hand, it is readily checked from (\ref{decompo-delha-j-n-2}) that
$$\norm{(\delha J^n)_{tus} }_{\cb_{\ga',p}} \leq c^6_{x,f} \lcl 1+\norm{\psi}_{\cb_{\ga',p}}\rcl \lln t-s \rrn^{\ga-(\ga'-\frac{1}{2})},$$
and therefore
$$\cn[\delha J^n;\cacha_3^{\ga-(\ga'-\frac{1}{2})}(\llbracket 0,t^n_{q+1}\rrbracket_n;\cb_{\ga',p})] \leq c^6_{x,f} \lcl 1+\norm{\psi}_{\cb_{\ga',p}} \rcl .$$
By using the estimate (\ref{contr-1}), we get
$$\cn[J^n;\cac_2^\ep(\llbracket 0,t^n_{q+1}\rrbracket_n ;\cb_{\ga',p})] \leq c^7_{x,f} \lcl 1+\norm{\psi}_{\cb_{\ga',p}}\rcl \lp  T_0^{3\ga-\mu}+T_0^{\ga+\ga'-\mu}+ T_0^{\ga-(\ga'-\frac{1}{2})-\ep} \rp.$$
It only remains to pick $T_0$ such that
$$c^7_{x,f}\lp  T_0^{3\ga-\mu}+T_0^{\ga+\ga'-\mu}+ T_0^{\ga-(\ga'-\frac{1}{2})-\ep} \rp \leq 1.$$
We can follow the same lines to show (\ref{contr-b-p-2}) from the estimate (\ref{contr-2}).

\smallskip

\noindent
It is now easy to realize that the same reasoning (with the same constants) can be applied on the interval $[T_0,2T_0]$ by replacing $\psi$ with $y^n_{T_0}$, and then on the interval $[2T_0,3T_0]$, etc.
\end{proof}

\begin{corollary}\label{coro:reco}
With the notation of Proposition \ref{prop:inter-exi}, there exists a constant $c_{x,f}$ such that for any $k$,
\begin{equation}\label{est-prol-1}
\cn[J^n;\cac_2^\mu(\llbracket kT_0,(k+2)T_0 \wedge 1 \rrbracket_n;\cb_p)] \leq \lcl 1+\norm{y^n_{(k+1)T_0}}_{\cb_{\ga',p}} \rcl+c_{x,f} \lcl 1+\norm{y^n_{kT_0}}_{\cb_{\ga',p}}\rcl,
\end{equation}
\begin{equation}\label{est-prol-2}
\cn[J^n;\cac_2^\ep(\llbracket kT_0,(k+2)T_0 \wedge 1 \rrbracket_n;\cb_{\ga',p})] \leq \lcl 1+\norm{y^n_{(k+1)T_0}}_{\cb_{\ga',p}} \rcl+c_{x,f} \lcl 1+\norm{y^n_{kT_0}}_{\cb_{\ga',p}}\rcl.
\end{equation}
\end{corollary}

\begin{proof}
If $kT_0 \leq s < (k+1)T_0 \leq t <(k+2)T_0$,
$$J^n_{ts}=J^n_{t,(k+1)T_0}-S_{t,(k+1)T_0} J^n_{(k+1)T_0,s}-(\delha J^n)_{t,(k+1)T_0,s}.$$
We already know that
$$\norm{J^n_{t,(k+1)T_0}}_{\cb_p}+\norm{J^n_{(k+1)T_0,s}}_{\cb_p} \leq \lln t-s \rrn^\mu  \lcl 2+\norm{y^n_{(k+1)T_0}}_{\cb_{\ga',p}}+\norm{y^n_{kT_0}}_{\cb_{\ga',p}}\rcl.$$
By using the decomposition (\ref{dec-delha-j-n}), together with the estimates (\ref{contr-norm-sob}), (\ref{est-2}) and (\ref{est-3}), we get $\norm{(\delha J^n)_{t,(k+1)T_0,s}}_{\cb_p} \leq c_x \lln t-s \rrn^\mu \lcl 1+\norm{y^n_{kT_0}}_{\cb_{\ga',p}} \rcl$, which yields (\ref{est-prol-1}). (\ref{est-prol-2}) can be shown with the same arguments.
\end{proof}

\begin{proof}[Proof of Theorem \ref{theo-exi}]
With the same estimates as in (\ref{contr-norm-sob}), we first deduce from Proposition \ref{prop:inter-exi}
$$\cn[y^n;\cac_1^0(\llbracket kT_0,(k+1)T_0 \wedge 1 \rrbracket_n;\cb_{\ga',p}] \leq c_{x,f}^1 \lcl 1+\norm{y^n_{kT_0}}_{\cb_{\ga',p}}\rcl ,$$
where the constant $c_{x,f}^1$ does not depend on $k$. As $T_0$ is independent of $y^n$, this leads to
\begin{equation}\label{exi-sol-contr-unif-1}
\cn[y^n;\cac_1^0(\llbracket 0,1\rrbracket_n;\cb_{\ga',p})] \leq c_{x,f}^2 \lcl 1+\norm{\psi}_{\cb_{\ga',p}} \rcl.
\end{equation}
From this uniform control, we get, by repeating the argument of Corollary \ref{coro:reco},
\begin{equation}\label{exi-sol-contr-unif-3}
\cn[J^n;\cac_2^\mu(\llbracket 0,1\rrbracket_n;\cb_p)] \leq c_{x,f}^4 \lcl 1+\norm{\psi}_{\cb_{\ga',p}} \rcl,
\end{equation}
and then
\begin{equation}\label{exi-sol-contr-unif-2}
\cn[y^n;\cacha_1^\ga(\llbracket 0,1\rrbracket_n;\cb_p)] \leq c_{x,f}^5 \lcl 1+\norm{\psi}_{\cb_{\ga',p}} \rcl.
\end{equation}
Now remember that $y^n$ is extended on $[0,1]$ by linear interpolation, and so
\bean
\cn[y^n;\cac_1^\ga([0,1];\cb_p)] & \leq & 3\cn[y^n;\cac_1^\ga(\llbracket 0,1\rrbracket_n;\cb_p)]\\
&\leq & 3\cn[y^n;\cacha_1^\ga(\llbracket 0,1\rrbracket_n;\cb_p)]+c_{\ga'} \cn[y^n;\cac_1^0(\llbracket 0,1\rrbracket_n;\cb_{\ga',p})] \\
&\leq & c_{x,f}^6 \lcl 1+\norm{\psi}_{\cb_{\ga',p}} \rcl.
\eean
Thus, we are in a position to apply the Ascoli Theorem and to assert the existence of a subsequence $y^{n_k}$ of $y^n$ that converges to an element $y$ in $\cac_1^0([0,1];\cb_p)$. It remains to check that $y$ is a solution of (\ref{equa-gene}). To do so, let $s<t \in [0,1]$ and consider two sequences $s_{n_k}<t_{n_k}\in \Pi^{n_k}$ such that $s_{n_k}$ decreases to $s$ and $t_{n_k}$ increases to $t$. Then
\begin{equation}\label{decom-j-y}
\norm{J^y_{ts}}_{\cb_p} \leq \norm{J^y_{ts}-J^{y^{n_k}}_{ts}}_{\cb_p}+\norm{J^{y^{n_k}}_{ts}-J^{y^{n_k}}_{t_{n_k}s_{n_k}}}_{\cb_p}+\norm{J^{y^{n_k}}_{t_{n_k}s_{n_k}}}_{\cb_p}.
\end{equation}
On the one hand,
$$\norm{J^y_{ts}-J^{y^{n_k}}_{ts}}_{\cb_p} \leq c_{x,f} \cn[y-y^{n_k};\cac_1^0([0,1];\cb_p)] \to 0,$$
while on the other hand
\begin{multline*}
\norm{J^{y^{n_k}}_{ts}-J^{y^{n_k}}_{t_{n_k}s_{n_k}}}_{\cb_p} \leq c_f \lcl \norm{X^{x,i}_{ts}-X^{x,i}_{t_{n_k}s_{n_k}}}_{\cl(\cb_p,\cb_p)}+ \norm{X^{xx,ij}_{ts}-X^{xx,ij}_{t_{n_k}s_{n_k}}}_{\cl(\cb_p,\cb_p)}\rcl\\
+c_{x,f} \lcl \norm{y^{n_k}_t-y^{n_k}_{t_{n_k}}}_{\cb_p}+\norm{y^{n_k}_{s_{n_k}}-y^{n_k}_{s}}_{\cb_p} \rcl,
\end{multline*}
from which we easily deduce, with the uniform controls (\ref{exi-sol-contr-unif-1}) and (\ref{exi-sol-contr-unif-2}) in mind, $$\norm{J^{y^{n_k}}_{ts}-J^{y^{n_k}}_{t_{n_k}s_{n_k}}}_{\cb_p}\to 0.$$
Finally, owing to (\ref{exi-sol-contr-unif-3}),
$$\norm{J^{y^{n_k}}_{t_{n_k}s_{n_k}}}_{\cb_p}\leq c^7_{x,f} \lcl 1+\norm{\psi}_{\cb_{\ga',p}} \rcl \lln t-s \rrn^\mu.$$
Going back to (\ref{decom-j-y}), this proves that $J^y \in \cac_3^\mu([0,1];\cb_p)$. Then we follow the same lines starting with the estimate
$\cn[J^n;\cac_3^\ep(\llbracket 0,1 \rrbracket_n;\cb_{\ga',p})] \leq c^4_{x,f} \lcl 1+\norm{\psi}_{\cb_{\ga',p}} \rcl$ to get $J^y \in  \cac_3^\ep([0,1];\cb_{\ga',p})$, and so $y$ is indeed a solution of (\ref{equa-gene}) in $\cb_{\ga',p}$.
\end{proof}

\section{Uniqueness of the solution}\label{sec:uni}

In this section, we mean to prove Theorem \ref{theo-uni}. Accordingly, we assume that $p>n$ and that Conditions (A1), (A2), (X)$_\ga$ and (F)$_3$ are all checked. Let $y$ be a solution of (\ref{equa-gene}) in $\cb_{\ga',p}$, for some (fixed) parameter $\ga'\in (1-\ga,1/2+\ga)$, with initial condition $\psi \in \cb_{\ga',p}$, and let $y^n$ be the path described by the scheme (\ref{schema}), with the same initial condition $\psi$.

\

We introduce, for all $s<t\in \Pi^n$, the quantity
\begin{multline*}
\cn[y-y^n;\cq(\llbracket s,t\rrbracket_n)]:=\\
\cn[y-y^n;\cacha_1^\ga(\llbracket s,t\rrbracket_n;\cb_p)]+\cn[y-y^n;\cac_1^0(\llbracket s,t\rrbracket_n;\cb_{\ga',p})]+\cn[K^y-K^{y^n};\cac_2^{2\ga}(\llbracket s,t\rrbracket_n ;\cb_p)].
\end{multline*}
Besides, we fix $\mu >1$, $\ep >0$ such that $\norm{J^y_{ts}}_{\cb_p} \leq c \lln t-s \rrn^\mu$ and $\norm{J^y_{ts}}_{\cb_{\ga',p}} \leq c \lln t-s \rrn^\ep$.

\smallskip

\noindent
The proof of Theorem \ref{theo-uni} is based on the two following preliminary results, which aim at controlling, as in the previous section, the residual term $J$:

\begin{lemma}\label{lem:base-uni}
For all $\tilde{\mu}>1$ and $\ka >0$, there exists two constants $c_y,c_{\tilde{\mu}}$ such that if $s<t \in \Pi^n$,
\begin{multline*}
\norm{J^y_{ts}-J^{y^n}_{ts}}_{\cb_{\ga',p}} \leq c_{y} \lcl \frac{1}{(2^n)^\ep}+\frac{1}{(2^n)^{\mu-1}} \rcl +c_{\tilde{\mu}} \lcl \lln t-s \rrn^\ka+\lln t-s \rrn^{\tilde{\mu}-\ga'} \rcl\\
\lcl \cn[\delha (J^y-J^{y^n});\cac_3^\ka(\llbracket s,t\rrbracket_n;\cb_{\ga',p})]+\cn[\delha (J^y-J^{y^n});\cac_3^{\tilde{\mu}}(\llbracket s,t\rrbracket_n;\cb_{p})] \rcl.
\end{multline*} 
$$\norm{J^y_{ts}-J^{y^n}_{ts}}_{\cb_{p}} \leq \frac{c_y \lln t-s \rrn}{(2^n)^{\mu-1}} +c_{\tilde{\mu}} \lln t-s \rrn^{\tilde{\mu}} \cn[\delha (J^y-J^{y^n});\cac_3^{\tilde{\mu}}(\llbracket s,t\rrbracket_n;\cb_{p})].$$
\end{lemma}

\begin{proof}
Going back to the notation of Subsection \ref{subsec:notations}, we decompose $J^y_{ts}-J^{y^n}_{ts}$ as 
$$J^y_{ts}-J^{y^n}_{ts}=\lc J^{y,\Pi^n}_{ts}-J^{y^n,\Pi^n}_{ts} \rc+R^{y,\Pi^n}_{ts},$$
with, if $s=t_k^n$ and $t=t_l^n$,
$$R^{y,\Pi^n}_{ts}:=J^y_{tt_{l-1}^n} +\sum_{i=k}^{l-2} S_{tt_{i+1}^n} J^y_{t_{i+1}^nt_i^n}.$$
To handle the term into brackets, we use the same arguments as in the proof of Lemma \ref{lemme:base}, which yield here
\begin{multline*}
\norm{J^{y,\Pi^n}_{ts}-J^{y^n,\Pi^n}_{ts}}_{\cb_{\ga',p}} \leq c_{\tilde{\mu},\ga'} \lcl \lln t-s \rrn^\ka+\lln t-s \rrn^{\tilde{\mu}-\ga'} \rcl\\
\lcl \cn[\delha (J^y-J^{y^n});\cac_3^\ka(\llbracket s,t\rrbracket_n;\cb_{\ga',p})]+\cn[\delha (J^y-J^{y^n});\cac_3^{\tilde{\mu}}(\llbracket s,t\rrbracket_n;\cb_{p})] \rcl
\end{multline*} 
and
$$\norm{J^{y,\Pi^n}_{ts}-J^{y^n,\Pi^n}_{ts}}_{\cb_{p}} \leq c_{\mu,\ga'} \lln t-s \rrn^{\tilde{\mu}} \cn[\delha (J^y-J^{y^n});\cac_3^{\tilde{\mu}}(\llbracket s,t\rrbracket_n;\cb_{p})].$$
Then it suffices to observe that
$$\norm{R^{y,\Pi^n}_{ts}}_{\cb_p} \leq \frac{c_y}{(2^n)^{\mu-1}} \lcl \lln t-t^n_{l-1} \rrn+\sum_{i=k}^{l-2} \lln t^n_{i+1}-t^n_i \rrn \rcl \leq \frac{c_y \lln t-s\rrn}{(2^n)^{\mu-1}}$$
and
\begin{equation}\label{utilite-sec-cond}
\norm{R^{y,\Pi^n}_{ts}}_{\cb_{\ga',p}} \leq \frac{c_y}{(2^n)^\ep}+\sum_{i=k}^{l-2} \lln t-t_{i+1}^n \rrn^{-\ga'} \frac{c_y}{(2^n)^\mu} \leq c_{y,\ga'} \lcl \frac{1}{(2^n)^\ep}+\frac{1}{(2^n)^{\mu-1}} \rcl.
\end{equation}
\end{proof}

\begin{lemma}\label{lem:uni-2}
Set $\tilde{\mu}:=\inf(\ga+\ga',3\ga)$. Then for all $s<t\in \Pi^n$,
\begin{equation}\label{contr-delha-j-y-j-n-1}
\cn[\delha (J^y-J^{y^n});\cac_3^{\tilde{\mu}}(\llbracket s,t\rrbracket_n;\cb_p)] \leq c_{y,x,f,\psi} \cn[y-y^n;\cq(\llbracket s,t\rrbracket_n )],
\end{equation}
\begin{equation}\label{contr-delha-j-y-j-n-2}
\cn[\delha (J^y-J^{y^n});\cac_3^{\ga}(\llbracket s,t\rrbracket_n;\cb_{\ga',p})] \leq c_{y,x,f,\psi} \cn[y-y^n;\cq(\llbracket s,t\rrbracket_n )].
\end{equation}
\end{lemma}

\begin{proof}
(\ref{contr-delha-j-y-j-n-1}) is a consequence of the decomposition (\ref{dec-delha-j-n}). Indeed, one has for instance, if $\cn_y:=\cn[y;\cacha_1^\ga([0,1];\cb_p)]+\cn[y;\cac_1^0([0,1];\cb_{\ga',p})]$,
\bean
\lefteqn{\norm{f_i'(y_s+r(\der y)_{us})-f_i'(y_s)-f_i'(y^n_s+r(\der y^n)_{us})+f_i'(y^n_s)}_{\cb_p}}\\
&\leq & \norm{ r\int_0^1 dr' \lc f_i''(y_s+rr'(\der y)_{us})-f_i''(y^n_s+rr' (\der y^n)_{us})\rc (\der y)_{us}}_{\cb_p}\\
& &+\norm{r\int_0^1 dr' f_i''(y^n_s+rr'(\der y^n)_{us}) \der (y-y^n)_{us}}_{\cb_p}\\
&\leq & c_f \cn_y \lln u-s\rrn^\ga \int_0^1 dr' \norm{(y_s+rr'(\der y)_{us})-(y^n_s+rr'(\der y^n)_{us})}_{\cb_\infty}\\
& &+c_f \lln u-s \rrn^\ga \cn[y-y^n;\cq(I)]\\
&\leq & c_{f,\cn_y} \lln u-s \rrn^\ga \cn[y-y^n;\cq(I)],
\eean
where we have used the continuous inclusion $ \cb_{\ga',p}\subset \cb_\infty$. As for (\ref{contr-delha-j-y-j-n-2}), it suffices to observe, with the expression (\ref{decompo-delha-j-n-2}) in mind, that for instance, due to the assumption (A2) and the control (\ref{contr-nem-2}), one has
\begin{eqnarray}
\lefteqn{\norm{X^{x,i}_{tu}(f_i(y_u)-f_i(y^n_u))}_{\cb_{\ga',p}}} \nonumber\\
&\leq & c_x \lln t-s\rrn^{\ga} \norm{f_i(y_u)-f_i(y^n_u)}_{\cb_{\ga',p}} \nonumber\\
&\leq & c_x \lln t-s \rrn^{\ga} \norm{\int_0^1 dr \, f_i'(y_u+r(y^n_u-y_u))(y^n_u-y_u)}_{\cb_{\ga',p}} \nonumber\\
&\leq & c_x \lln t-s \rrn^{\ga} \norm{y^n_u-y_u}_{\cb_{\ga',p}} \norm{\int_0^1 dr \, f_i'(y_u+r(y^n_u-y_u))}_{\cb_{\ga',p}}\nonumber\\
&\leq & c_{x,f,\cn_y,\cn_{y^n}} \lln t-s \rrn^{\ga} \cn[y-y^n;\cq(I)]\nonumber\\
&\leq & c_{x,f,\cn_y,\psi} \lln t-s \rrn^{\ga} \cn[y-y^n;\cq(I)],\label{exemp}
\end{eqnarray}
where, to get the last estimate, we have appealed to the uniform control $\cn_{y^n} \leq c_{x,f,\psi}$ established in the proof of Theorem \ref{theo-exi}.
\end{proof}

\begin{proof}[Proof of Theorem \ref{theo-uni}]
Let $T_1 \leq 1 \in \Pi^n$. Write 
$$\delha(y-y^n)_{ts}=X^{x,i}_{ts} \lc f_i(y_s)-f_i(y^n_s)\rc+X^{xx,ij}_{ts} \lc F_{ij}(y_s)-F_{ij}(y^n_s)\rc+\lc J^y_{ts}-J^{y^n}_{ts}\rc,$$
and use the two previous lemmas to deduce first
$$\cn[y-y^n;\cacha_1^\ga(\llbracket 0,T_1 \rrbracket_n;\cb_p)] \leq c_{y,x,f,\psi} T_1^\ga \cn[y-y^n;\cq(\llbracket 0,T_1\rrbracket_n )]+\frac{c_y}{(2^n)^{\mu-1}}$$
and secondly
\begin{multline*}
\cn[y-y^n;\cac_1^0(\llbracket 0,T_1 \rrbracket_n;\cb_{\ga',p})] \leq c_{y,x,f,\psi} T_1^{\ga} \cn[y-y^n;\cq(\llbracket 0,T_1\rrbracket_n )]\\
+c_{y} \lcl \frac{1}{(2^n)^\ep}+\frac{1}{(2^n)^{\mu-1}} \rcl.
\end{multline*}
Then, since $K^y_{ts}-K^{y^n}_{ts}=X^{xx,ij}_{ts} \lc F_{ij}(y_s)-F_{ij}(y^n_s) \rc +\lc J^y_{ts}-J^{y^n}_{ts}\rc$, one has
$$\cn[K^y-K^{y^n};\cac_2^{2\ga}(\llbracket 0,T_1 \rrbracket_n;\cb_p)] \leq c_{y,x,f,\psi} T_1^\ga \cn[y-y^n;\cq(\llbracket 0,T_1\rrbracket_n )]+\frac{c_y}{(2^n)^{\mu-1}}$$
and we have thus proved that
$$
\cn[y-y^n;\cq(\llbracket 0,T_1 \rrbracket_n)] \leq c^1_{y,x,f,\psi} T_1^{\ga} \cn[y-y^n;\cq(\llbracket 0,T_1\rrbracket_n )]+c^1_{y} \lcl \frac{1}{(2^n)^\ep}+\frac{1}{(2^n)^{\mu-1}} \rcl.
$$
Choose $T_1$ such that $c^1_{y,x,f,\psi} T_1^{\ga}=\frac{1}{2}$ to obtain
$$\cn[y-y^n;\cq(\llbracket 0,T_1\rrbracket_n)] \leq 2c^1_{y} \lcl \frac{1}{(2^n)^\ep}+\frac{1}{(2^n)^{\mu-1}} \rcl.$$
By using the same arguments on $\llbracket kT_1,(k+1)T_1 \rrbracket_n$, we get
$$\cn[y-y^n;\cq(\llbracket kT_1,(k+1)T_1\rrbracket_n)] \leq 2c^1_{y} \lcl \frac{1}{(2^n)^\ep}+\frac{1}{(2^n)^{\mu-1}} \rcl+c_{x,f} \norm{y_{kT_1}-y^n_{kT_1}}_{\cb_{\ga',p}},$$
and it is now easy to establish that
\begin{equation}\label{contr-uni-norm}
\cn[y-y^n;\cacha_1^\ga(\llbracket 0,1\rrbracket_n;\cb_p)]+\cn[y-y^n;\cac_1^0(\llbracket 0,1\rrbracket_n;\cb_{\ga',p})] \leq c_{y,x,f,\psi} \lcl \frac{1}{(2^n)^\ep}+\frac{1}{(2^n)^{\mu-1}} \rcl.
\end{equation}
This inequality clearly proves the uniqueness of the solution and therefore, it enables us to identify $y$ with the solution constructed in Section \ref{sec:existence}. This identification allows in turn to choose $\mu$ and $\ep$ as in Proposition \ref{prop:inter-exi} and to assert that $\cn_y \leq c_{x,f,\psi}$, which completes the proof.

\end{proof}

\section{Continuity of the solution}\label{sec:conti}

It remains to prove Theorem \ref{theo-conti}. In accordance with the statement of this result, we suppose that $p>n$ and that Assumptions (A1), (A2), (X)$_\ga$ and (F)$_3$ are all satisfied. We fix $\ga' \in (1-\ga,\ga+1/2)$ and the two initial conditions $\psi,\tilde{\psi} \in \cb_{\ga',p}$. We denote by $X=(X^x,X^{ax},X^{xx})$ (resp. $\tilde{X}=(\tilde{X}^x,\tilde{X}^{ax},\tilde{X}^{xx})$) the path constructed from $(x,\xrgh^\bt)$ (resp. $(\tilde{x},\tilde{\xrgh}^\bt))$ through Definition \ref{defi-chemin}. With this notation, we define $y^n$ as the path described by the scheme (\ref{schema}) and $\yti^n$ as the path obtained by replacing $(\psi,X^x,X^{xx})$ with $(\tilde{\psi},\tilde{X}^x,\tilde{X}^{xx})$ in the latter scheme.

\smallskip

\noindent
Besides, we define $\Jti$ and $\Kti$ by replacing $(X^x,X^{xx})$ with $(\Xti^x,\Xti^{xx})$ in Formulas (\ref{defi-j}) and (\ref{defi-k}). For the sake of clarity, we also set $J^n:=J^{y^n}$, $K^n:=K^{y^n}$, $\Jti^n:=\Jti^{y^n}$, $\Kti^n=\Kti^{\yti^n}$, and as in the previous section, we introduce the intermediate quantity
\begin{multline*}
\cn[y^n-\yti^n;\tilde{Q}(\llbracket s,t\rrbracket_n)]\\
:=\cn[y^n-\yti^n;\cacha_1^\ga(\llbracket s,t \rrbracket_n;\cb_p)]+\cn[y^n-\yti^n;\cac_1^0(\llbracket s,t \rrbracket_n;\cb_{\ga',p})]+\cn[K^n-\Kti^n;\cac_2^{2\ga}(\llbracket s,t\rrbracket_n;\cb_p)].
\end{multline*}

Remember that owing to the results of Section \ref{sec:existence}, we can rely on the uniform control
$$\cn[y^n;\cacha_1^\ga(\llbracket 0,1 \rrbracket_n;\cb_p)]+\cn[y^n;\cac_1^0(\llbracket 0,1 \rrbracket_n;\cb_{\ga',p})]+\cn[K^n;\cac_2^{2\ga}(\llbracket 0,1\rrbracket_n;\cb_p)]\leq c_{x,\psi},$$
with an equivalent result for $\yti^n$. The proof of Theorem \ref{theo-conti} now leans on the two following lemmas:

\begin{lemma}\label{lem:base-conti}
For all $\tilde{\mu} >1$ and $\ka >0$, there exists a constant $c=c_{\tilde{\mu},\ka}$ such that if $s<t\in \Pi^n$,
\begin{multline*}
\norm{J^n_{ts}-\Jti^n_{ts}}_{\cb_{\ga',p}} \leq c \lcl \lln t-s \rrn^\ka+\lln t-s \rrn^{\tilde{\mu}-\ga'} \rcl\\
\lcl \cn[\delha (J^n-\Jti^n);\cac_3^\ka(\llbracket s,t\rrbracket_n;\cb_{\ga',p})]+\cn[\delha (J^n-\Jti^n);\cac_3^{\tilde{\mu}}(\llbracket s,t\rrbracket_n;\cb_{p})] \rcl
\end{multline*} 
and
$$\norm{J^n_{ts}-\Jti^n_{ts}}_{\cb_{p}} \leq c \lln t-s \rrn^{\tilde{\mu}} \cn[\delha (J^n-\Jti^n);\cac_3^{\tilde{\mu}}(\llbracket s,t\rrbracket_n;\cb_{p})].$$
\end{lemma}

\begin{proof}
It suffices to follow the lines of the proof of Lemma \ref{lemme:base}.
\end{proof}

\begin{lemma}
Set $\tilde{\mu}:=\inf(\ga+\ga',3\ga)$. Then for all $s<t\in \Pi^n$,
\begin{equation}\label{lem:conti-2-1}
\cn[\delha(J^n-\Jti^n);\cac_3^{\tilde{\mu}}(\llbracket s,t\rrbracket_n;\cb_p)] \leq c_{x,\xti,\psi,\tilde{\psi}} \lcl \cn[y^n-\yti^n;\tilde{Q}(\llbracket s,t\rrbracket_n)]+ \norm{\xrgh-\tilde{\xrgh}}_\ga \rcl
\end{equation}
and
\begin{equation}\label{lem:conti-2-2}
\cn[\delha(J^n-\Jti^n);\cac_3^{\ga}(\llbracket s,t\rrbracket_n;\cb_{\ga',p})] \leq c_{x,\xti,\psi,\tilde{\psi}} \lcl \cn[y^n-\yti^n;\tilde{Q}(\llbracket s,t\rrbracket_n)]+ \norm{\xrgh-\tilde{\xrgh}}_\ga \rcl.
\end{equation}
\end{lemma}

\begin{proof}
This is the same type of arguments as in the proof of Lemma \ref{lem:uni-2}. For (\ref{lem:conti-2-1}), resort to the decomposition (\ref{dec-delha-j-n}) and notice for instance that
\bean
\lefteqn{\norm{X^{x,i}_{tu} \lp \int_0^1 dr \, f_i'(y^n_s+r(\der y^n)_{us}) \cdot K^n_{us} \rp -\Xti^{x,i}_{tu} \lp \int_0^1 dr \, f_i'(\yti^n_s+r(\der \yti^n)_{us}) \cdot \Kti^n_{us} \rp}_{\cb_p}}\\
&\leq & c\,  \norm{X^{x,i}_{tu}-\Xti^{x,i}_{tu}}_{\cl(\cb_p,\cb_p)} \norm{K^n_{us}}_{\cb_p}+\norm{\Xti^{x,i}_{tu}}_{\cl(\cb_p,\cb_p)} \\
& & \norm{\int_0^1 dr \, f_i'(y^n_s+r(\der y^n)_{us}) \cdot K^n_{us}-\int_0^1 dr \, f_i'(\yti^n_s+r(\der \yti^n)_{us}) \cdot \Kti^n_{us}}_{\cb_p}\\
&\leq & c_{x,\xti,\psi} \lln t-s \rrn^{3\ga}\norm{\xrgh-\tilde{\xrgh}}_\ga+c_{\xti} \lln t-u \rrn^\ga\\
& & \bigg\{ \norm{\int_0^1 dr \, \lc f_i'(y^n_s+r(\der y^n)_{us})-f_i'(\yti^n_s+r(\der \yti^n)_{us})\rc \cdot K^n_{us} }_{\cb_p} \\
& & +\norm{\int_0^1 dr \, f_i'(\yti^n_s+r(\der \yti^n)_{us}) \cdot \lc K^n_{us}-\Kti^n_{us} \rc }_{\cb_p} \bigg\}\\
&\leq & c^1_{x,\xti,\psi} \lln t-s \rrn^{3\ga} \norm{\xrgh-\tilde{\xrgh}}_\ga+c^2_{x,\xti,\psi} \lln t-s \rrn^{3\ga} \cn[y^n-\yti^n;\tilde{Q}(\llbracket s,t\rrbracket_n)],
\eean
where we have used the continuous inclusion $\cb_{\ga',p} \subset \cb_\infty$. (\ref{lem:conti-2-2}) can be proved likewise, with the same kind of estimates as in the proof of (\ref{exemp}).

\end{proof}

\begin{proof}[Proof of Theorem \ref{theo-conti}]
By following the same procedure as in the proof of Theorem \ref{theo-uni}, we first deduce
\begin{multline*}
\cn[y^n-\yti^n;\tilde{Q}(\llbracket 0,T_2\rrbracket_n)] \\
\leq c^1_{x,\xti,\psi,\tilde{\psi}} \lcl T_2^\ga \cn[y^n-\yti^n;\tilde{Q}(\llbracket 0,T_2 \rrbracket_n)] +\norm{\psi-\tilde{\psi}}_{\cb_{\ga',p}}+\norm{\xrgh-\tilde{\xrgh}}_\ga \rcl.
\end{multline*}
Indeed, one has for instance, if $0\leq s<t\leq T_2$,
\bean
\norm{X^{x,i}_{ts} \lc f_i(y^n_s)-f_i(\yti^n_s)\rc }_{\cb_p} &\leq & c_x \lln t-s \rrn^\ga \norm{y^n_s-\yti^n_s}_{\cb_p}\\
&\leq & c_x \lln t-s \rrn^\ga \lcl \norm{\delha (y^n-\yti^n)_{s0}}_{\cb_p}+\norm{\psi-\tilde{\psi}}_{\cb_{\ga',p}} \rcl\\
&\leq & c_x \lln t-s \rrn^\ga \lcl T_2^\ga \cn[y^n-\yti^n;\tilde{Q}(\llbracket 0,T_2 \rrbracket_n)]+\norm{\psi-\tilde{\psi}}_{\cb_{\ga',p}} \rcl.
\eean
Then we take $T_2$ such that $c^1_{x,\xti,\psi,\tilde{\psi}}  T_2^\ga=\frac{1}{2}$ so as to retrieve
$$\cn[y^n-\yti^n;\tilde{Q}(\llbracket 0,T_2\rrbracket_n)] 
\leq 2 \, c^1_{x,\xti,\psi,\tilde{\psi}} \lcl  \norm{\psi-\tilde{\psi}}_{\cb_{\ga',p}}+\norm{\xrgh-\tilde{\xrgh}}_\ga \rcl.$$
Repeating the procedure on $[T_2,2T_2]$, $[2T_2,3T_2]$,..., leads to the uniform control
\begin{multline}\label{ito-discret}
\cn[y^n-\yti^n;\cacha_1^\ga(\llbracket 0,1\rrbracket_n;\cb_p)]+\cn[y^n-\yti^n;\cac_1^0(\llbracket 0,1\rrbracket_n;\cb_{\ga',p})]\\
\leq c_{x,\xti,\psi,\tilde{\psi}} \lcl \norm{\psi-\tilde{\psi}}_{\cb_{\ga',p}}+\norm{\xrgh-\tilde{\xrgh}}_\ga \rcl.
\end{multline}
To conclude with, let us introduce, for all $s<t\in [0,1]$, two sequences $s_n<t_n \in \Pi^n$ such that $s_n$ decreases to $s$ and $t_n$ increases to $t$, and write (for instance) successively
$$\norm{\delha(y-\yti)_{ts}}_{\cb_p} \leq \norm{\delha(y-\yti)_{tt_n}}_{\cb_p}+\norm{\delha(y-\yti)_{t_ns_n}}_{\cb_p}+\norm{\delha(y-\yti)_{s_ns}}_{\cb_p},$$
$$\norm{\delha(y-\yti)_{t_ns_n}}_{\cb_p} \leq \norm{\delha(y-y^n)_{t_ns_n}}_{\cb_p}+\norm{\delha(y^n-\yti^n)_{t_ns_n}}_{\cb_p}+\norm{\delha(\yti-\yti^n)_{t_ns_n}}_{\cb_p}.$$
The control (\ref{ito-discret}), together with the approximation result (\ref{contr-uni-norm}), then provides (\ref{ito}).

\end{proof}

\section{Appendix A: a useful algorithm}\label{sec:algo}

We give here the description and an analysis of the algorithm used in the proofs of Lemmas \ref{lemme:base}, \ref{lem:base-uni} and \ref{lem:base-conti}.

\smallskip

\noindent
Consider a generic partition $\{ 0,1,2,\ldots,N\}$. We remove the inner points of this partition ($\{ 1,2,\ldots,N-1\}$) one by one according to the following procedure (see Figure 1):
\begin{list}{\labelitemi}{\leftmargin=1em\itemsep=1em}
\item At step 1, we successively remove, from the right to the left, every two points, starting from $N$ (excluded) until $0$ (also excluded). Then, still at step 1, we take off the point of the (updated) partition between $0$ (excluded) and the last removed point, if such a middle point exists.
\item We repeat the procedure with the remaining points (steps 2,3,...) until the partition is empty.
\end{list}

\

\noindent
We denote by:
\begin{itemize}
\item $M$ the number of steps necessary to empty the partition.
\item $(k_m)_{m\in \{1,\ldots,N-1\}}$ the sequence of successively removed points.
\item $k_m^+$ the point of the partition (at 'time' $m$ of the algorithm) that follows $k_m$ (when reading from the left to the right), $k_m^-$ the point that precedes it.
\item $A_r$ the total number of points that have been taken off at the end of step $r$. We also set $A_0:=0$.
\end{itemize}

\

\

\begin{center}
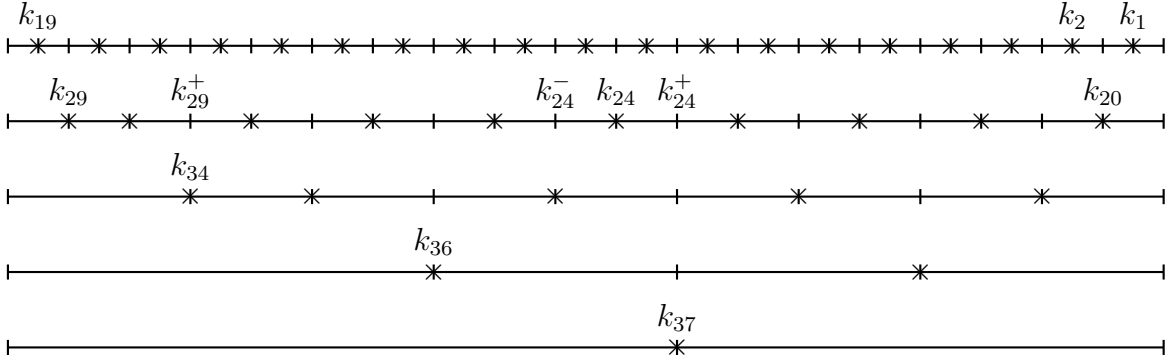
\begin{figure}[!ht]\label{schema-algo-points}
\begin{pspicture}(0,0)(15.2,6)
\psline(0,5)(15.2,5)
\psline(0,4.9)(0,5.1)  \psline(0.4,4.9)(0.4,5.1)  \psline(0.8,4.9)(0.8,5.1)  \psline(1.2,4.9)(1.2,5.1)   \psline(1.6,4.9)(1.6,5.1)  \psline(2,4.9)(2,5.1)  \psline(2.4,4.9)(2.4,5.1)  \psline(2.8,4.9)(2.8,5.1)  \psline(3.2,4.9)(3.2,5.1)  \psline(3.6,4.9)(3.6,5.1)  \psline(4,4.9)(4,5.1)  \psline(4.4,4.9)(4.4,5.1)  \psline(4.8,4.9)(4.8,5.1)  \psline(5.2,4.9)(5.2,5.1)  \psline(5.6,4.9)(5.6,5.1)  \psline(6,4.9)(6,5.1)  \psline(6.4,4.9)(6.4,5.1)  \psline(6.8,4.9)(6.8,5.1)  \psline(7.2,4.9)(7.2,5.1)  \psline(7.6,4.9)(7.6,5.1)  \psline(8,4.9)(8,5.1) \psline(8.4,4.9)(8.4,5.1) \psline(8.8,4.9)(8.8,5.1) \psline(9.2,4.9)(9.2,5.1) \psline(9.6,4.9)(9.6,5.1) \psline(10,4.9)(10,5.1) \psline(10.4,4.9)(10.4,5.1) \psline(10.8,4.9)(10.8,5.1) \psline(11.2,4.9)(11.2,5.1) \psline(11.6,4.9)(11.6,5.1) \psline(12,4.9)(12,5.1) \psline(12.4,4.9)(12.4,5.1) \psline(12.8,4.9)(12.8,5.1) \psline(13.2,4.9)(13.2,5.1) \psline(13.6,4.9)(13.6,5.1) \psline(14,4.9)(14,5.1) \psline(14.4,4.9)(14.4,5.1) \psline(14.8,4.9)(14.8,5.1) \psline(15.2,4.9)(15.2,5.1) 
\rput(14.8,5){$\times$}\rput(14,5){$\times$}\rput(13.2,5){$\times$}\rput(12.4,5){$\times$}\rput(11.6,5){$\times$}\rput(10.8,5){$\times$}\rput(10,5){$\times$}\rput(9.2,5){$\times$}\rput(8.4,5){$\times$}\rput(7.6,5){$\times$}\rput(6.8,5){$\times$}\rput(6,5){$\times$}\rput(5.2,5){$\times$}\rput(4.4,5){$\times$}\rput(3.6,5){$\times$}\rput(2.8,5){$\times$}\rput(2,5){$\times$}\rput(1.2,5){$\times$}\rput(0.4,5){$\times$}
\rput(14.8,5.4){$k_1$}\rput(14,5.4){$k_2$}\rput(0.4,5.4){$k_{19}$}

\psline(0,4)(15.2,4)
\psline(0,3.9)(0,4.1)    \psline(0.8,3.9)(0.8,4.1)     \psline(1.6,3.9)(1.6,4.1)   \psline(2.4,3.9)(2.4,4.1)    \psline(3.2,3.9)(3.2,4.1)    \psline(4,3.9)(4,4.1)   \psline(4.8,3.9)(4.8,4.1)    \psline(5.6,3.9)(5.6,4.1)    \psline(6.4,3.9)(6.4,4.1)   \psline(7.2,3.9)(7.2,4.1)   \psline(8,3.9)(8,4.1)  \psline(8.8,3.9)(8.8,4.1)  \psline(9.6,3.9)(9.6,4.1)  \psline(10.4,3.9)(10.4,4.1)  \psline(11.2,3.9)(11.2,4.1)  \psline(12,3.9)(12,4.1)  \psline(12.8,3.9)(12.8,4.1)  \psline(13.6,3.9)(13.6,4.1)  \psline(14.4,3.9)(14.4,4.1)  \psline(15.2,3.9)(15.2,4.1) 
\rput(14.4,4){$\times$}\rput(12.8,4){$\times$}\rput(11.2,4){$\times$}\rput(9.6,4){$\times$}\rput(8,4){$\times$}\rput(6.4,4){$\times$}\rput(4.8,4){$\times$}\rput(3.2,4){$\times$}\rput(1.6,4){$\times$}\rput(0.8,4){$\times$}
\rput(14.4,4.4){$k_{20}$}\rput(0.8,4.4){$k_{29}$}\rput(8,4.4){$k_{24}$}\rput(8.8,4.4){$k_{24}^+$}\rput(7.2,4.4){$k_{24}^-$}\rput(2.4,4.4){$k_{29}^+$}

\psline(0,3)(15.2,3)
\psline(0,2.9)(0,3.1)           \psline(2.4,2.9)(2.4,3.1)       \psline(4,2.9)(4,3.1)      \psline(5.6,2.9)(5.6,3.1)       \psline(7.2,2.9)(7.2,3.1)     \psline(8.8,2.9)(8.8,3.1)   \psline(10.4,2.9)(10.4,3.1)    \psline(12,2.9)(12,3.1)    \psline(13.6,2.9)(13.6,3.1)    \psline(15.2,2.9)(15.2,3.1) 
\rput(2.4,3){$\times$}\rput(4,3){$\times$}\rput(7.2,3){$\times$}\rput(10.4,3){$\times$}\rput(13.6,3){$\times$}
\rput(2.4,3.4){$k_{34}$}

\psline(0,2)(15.2,2)
\psline(0,1.9)(0,2.1)                \psline(5.6,1.9)(5.6,2.1)        \psline(8.8,1.9)(8.8,2.1)     \psline(12,1.9)(12,2.1)       \psline(15.2,1.9)(15.2,2.1) 
\rput(5.6,2){$\times$} \rput(12,2){$\times$}
\rput(5.6,2.4){$k_{36}$}

\psline(0,1)(15.2,1)
\psline(0,0.9)(0,1.1)                       \psline(8.8,0.9)(8.8,1.1)           \psline(15.2,0.9)(15.2,1.1) 
\rput(8.8,1){$\times$}
\rput(8.8,1.4){$k_{37}$}
\end{pspicture}
\caption{The algorithm for $N=38$. Each line corresponds to one step. Thus, $M=5$, $A_1=19$, $A_2=29$, $A_3=34$, $A_4=36$.}
\end{figure}
\end{center}

\

\begin{lemma}\label{lem:prel-a-r}
For every $r\in \{0,1,\ldots,M\}$, 
$$0\leq A_r-N\lp 1-\frac{1}{2^r}\rp \leq 1.$$
In particular, $\lln A_r-A_{r-1}-\frac{N}{2^r}\rrn \leq 1$ and $2^{M-1}\leq N \leq 2^{M+1}$.
\end{lemma}

\begin{proof}
This stems from a straightforward iteration procedure based on the formula $A_{r+1}=A_r+\left\lfloor \frac{N-A_r+1}{2}\right\rfloor$, $r\in \{0,1,\ldots,M-1\}$, where $\left\lfloor .\right\rfloor$ stands for the integer part. 
\end{proof}

\begin{proposition}\label{resu-algo}
Suppose that $\mu >1$, $0<\ga'<1$ and $\ka >0$. Then
\begin{equation}\label{contr-unif}
\sum_{r=1}^{M-1} \lcl \lln 1-\frac{k_{A_{r-1}+1}^-}{N} \rrn^{\ka}+\frac{1}{N^\mu}\sum_{m=A_{r-1}+2}^{A_{r}} \lln 1-\frac{k_m^+}{N}\rrn^{-\ga'} \lln k_m^+-k_m^- \rrn^\mu \rcl \leq c_{\ka,\mu,\ga'},
\end{equation}
for some finite constant $c_{\ka,\mu,\ga'}$ independent of $N$.

\end{proposition}

\begin{proof}
Actually, we use the following explicit expressions: at step $r$ ($r\in \{1,\ldots,M-1\}$), if $N-A_{r-1}$ is even, one has, for every $m\in \llbracket A_{r-1}+1,A_r-1 \rrbracket$,
\begin{equation}\label{k-m-plus}
k_m^+=N-2^r(m-A_{r-1})+2^r,
\end{equation}
\begin{equation}\label{k-m-moins}
k_m^-=N-2^r(m-A_{r-1}),
\end{equation}
and $k_{A_r}^+=N-2^r(A_r-A_{r-1})+2^r$, $k_{A_r}^-=0$, while if $N-A_{r-1}$ is odd, Formulas (\ref{k-m-plus}) and (\ref{k-m-moins}) remain true for $m\in \llbracket A_{r-1}+1,A_r-1 \rrbracket$, but $k_{A_r}^-=0$, $k_{A_r}^+=k_{A_r-1}^+=N-2^r(A_r-A_{r-1}-1)+2^r$. From these expressions, we first deduce
$$\sum_{r=1}^{M-1} \lln 1-\frac{k_{A_{r-1}+1}^-}{N} \rrn^\ka=\frac{1}{N^\ka} \sum_{r=1}^{M-1} (2^r)^\ka \leq c_\ka^1 \lp \frac{2^M}{N}\rp^\ka \leq c^2_\ka,$$
according to Lemma \ref{lem:prel-a-r}. Then, if $N-A_{r-1}$ is even, one has
\bean
\lefteqn{\sum_{m=A_{r-1}+2}^{A_r} \lln 1-\frac{k_m^+}{N}\rrn^{-\ga'} \lln k_m^+-k_m^-\rrn^\mu}\\
&=& \sum_{m=A_{r-1}+2}^{A_r-1} \lln 1-\frac{k_m^+}{N}\rrn^{-\ga'} \lln k_m^+-k_m^-\rrn^\mu+\lln 1-\frac{k_{A_r}^+}{N}\rrn^{-\ga'} \lln k_{A_r}^+\rrn^\mu\\
&=& \frac{2^{r(\mu-\ga')}}{N^{-\ga'}}\sum_{m=1}^{A_r-A_{r-1}-2} m^{-\ga'}+\lln 1-\frac{k_{A_r}^+}{N}\rrn^{-\ga'} \lln k_{A_r}^+\rrn^\mu\\
&\leq & c^3_{\ga'} \frac{(2^r)^{\mu-\ga'}}{N^{-\ga'}} (A_r-A_{r-1}-2)^{1-\ga'}\\
 & & +\frac{(2^r)^{-\ga'}}{N^{-\ga'}} (A_r-A_{r-1}-1)^{-\ga'} (N-2^r(A_r-A_{r-1}-1))^\mu\\
  &\leq & c^3_{\ga'} \frac{(2^r)^{\mu-\ga'}}{N^{-\ga'}} (A_r-A_{r-1}-2)^{1-\ga'}+\frac{(2^r)^{-\ga'}}{N^{-\ga'}} (N-2^r(A_r-A_{r-1}-1))^\mu.
\eean
since $A_r-A_{r-1} \geq 2$. In the same way, if $N-A_{r-1}$ is odd, one has
\bean
\lefteqn{\sum_{m=A_{r-1}+2}^{A_r} \lln 1-\frac{k_m^+}{N}\rrn^{-\ga'} \lln k_m^+-k_m^-\rrn^\mu}\\
 &\leq & c^3_{\ga'} \frac{(2^r)^{\mu-\ga'}}{N^{-\ga'}} (A_r-A_{r-1}-2)^{1-\ga'}\\
 & & +\frac{(2^r)^{-\ga'}}{N^{-\ga'}} (A_r-A_{r-1}-2)^{-\ga'} (N-2^r(A_r-A_{r-1}-2))^\mu\\
 &\leq & c^3_{\ga'} \frac{(2^r)^{\mu-\ga'}}{N^{-\ga'}} (A_r-A_{r-1}-2)^{1-\ga'}+\frac{(2^r)^{-\ga'}}{N^{-\ga'}} (N-2^r(A_r-A_{r-1}-2))^\mu.
\eean
since, in that case, $A_r-A_{r-1} \geq 3$. Thanks to Lemma \ref{lem:prel-a-r}, we now easily deduce
$$\frac{1}{N^\mu} \sum_{r=1}^{M-1} \sum_{m=A_{r-1}+2}^{A_r} \lln 1-\frac{k_m^+}{N}\rrn^{-\ga'} \lln k_m^+-k_m^-\rrn^\mu \leq \frac{c^3_{\ga'}}{N^{\mu-1}} \sum_{r=1}^{M-1} (2^r)^{\mu-1}+\frac{c^4_\mu}{N^{\mu-\ga'}} \sum_{r=1}^{M-1}(2^r)^{\mu-\ga'} \leq c_{\mu,\ga'}.$$
\end{proof}

\section{Appendix B}

This section is devoted to the proof of Proposition \ref{prop:cas-brown}. To this end, we will resort to the two following lemmas, respectively borrowed from \cite{GT} and \cite{brz-convol}:

\begin{lemma}\label{lem-grr}
Fix a time $T>0$. For all $\al,\be \geq 0$, $p,q \geq 1$, there exists a constant $c$ such that for every $R\in \cac_2([0,T];\cb_{\al,p})$,
$$\cn[R;\cac_2^\be([0,T];\cb_{\al,p})] \leq c \lcl U_{\be+\frac{2}{q},q,\al,p}(R)+\cn[\delha R;\cac_3^\be([0,T];\cb_{\al,p}] \rcl,$$
where
$$U_{\be,q,\al,p}(R)=\lc \int_{0\leq u<v\leq T} \lp \frac{\norm{R_{vu}}_{\cb_{\al,p}}}{\lln v-u\rrn^\be} \rp^q du dv \rc^{1/q}.$$
\end{lemma}

\begin{lemma}\label{lem:bdg}
For every $p\geq 2$, the Burkholder-Davies-Gundy inequality holds in $\cb_p$. In other words, for every $T>0$, if $B$ is a one-dimensional Brownian motion defined on complete filtered probability space $(\Omega,\mathcal{F},P)$ and $H$ is an adapted process with values in $L^2([0,T];\cb_p)$, then for any $q \geq 2$, there exists a constant $c$ independent of $H$ such that
\begin{equation}\label{bdg}
E\lc \sup_{0\leq t\leq T} \Big \| \int_0^t H_u \, dB_u\Big \| _{\cb_p}^q \rc \leq c\,  E \lc \lp \int_0^T \norm{H_u}_{\cb_p}^2 \, du \rp^{q/2} \rc.
\end{equation}
\end{lemma}

\

\

\begin{proof}[Proof of Proposition \ref{prop:cas-brown}]
On the whole, this is the same identification procedure as in the proof of Proposition \ref{prop:cas-regu}. The only difference lies in the fact that the direct estimates of the integrals under consideration will here be replaced with a joint use of Lemmas \ref{lem-grr} and \ref{lem:bdg}.

\smallskip

We denote by $y$ the (Itô) solution of (\ref{equa-gene}), with initial condition $\psi \in \cb_{\eta,p}$. Let us fix $\ga \in (1/3,1/2)$ such that $\ga+\eta >1$ and $2\ga >\eta$. If one refers to \cite[Theorem 1]{jent-roeck}, one can assert that $y\in \cac_1^0([0,1];\cb_{\eta,p})$ a.s, and one even knows that $\sup_{t\in [0,1]} E\lc \norm{y_t}_{\cb_{\eta,p}}^q \rc < \infty$ for every $q\in \N$. Then, since $(\delha y)_{ts}=\int_s^t S_{tu} \, dx^i_u \, f_i(y_u)$, one has, thanks to Lemma \ref{lem:bdg},
\begin{eqnarray}
E\lc \norm{ \int_s^t S_{tu} \, dx^i_u \, f_i(y_u)}_{\cb_p}^q \rc &\leq & c \, E\lc \lp \int_s^t \norm{S_{tu} f_i(y_u)}_{\cb_p}^2 \, du \rp^{q/2} \rc \nonumber\\
&\leq & c \lln t-s \rrn^{q/2-1} \int_s^t E\lc \norm{S_{tu}f_i(y_u)}_{\cb_p}^{q} \rc \, du \nonumber\\
& \leq & c \lln t-s \rrn^{q/2},\label{esti-preu-br}
\end{eqnarray}
and consequently, with the notation of Lemma \ref{lem-grr},
\bean
E\lc U_{\ga+\frac{2}{q},q,0,p}(\delha y)\rc &\leq & \lp \iint_{0\leq u<v\leq 1} \frac{E\lc \norm{(\delha y)_{vu}}_{\cb_p}^q \rc}{\lln v-u\rrn^{\ga q+2}} \, dudv \rp^{1/q}\\
&\leq & \lp \iint_{0\leq u<v \leq 1} \lln v-u\rrn^{q(\frac{1}{2}-\ga)-2} dudv \, \rp^{1/q} \ < \infty
\eean
by picking $q >1/(\frac{1}{2}-\ga)$. Together with the result of Lemma \ref{lem-grr}, this yields $y\in \cacha_1^\ga([0,1];\cb_p)$ a.s.

\smallskip

\noindent
As far as $K^y$ is concerned, we already know that $\delha K^y=X^{x,i} \der(f_i(y))$, which leads to $\delha K^y \in \cac_3^{2\ga}([0,1];\cb_p)$ a.s. Then, from the expression $K^y_{ts}=\int_s^t S_{tu} \, dx^i_u \, \der(f_i(y))_{us}$, we deduce, as in (\ref{esti-preu-br}), $E\lc \norm{K^y_{ts}}_{\cb_p}^q \rc \leq c \lln t-s \rrn^q$, and accordingly, thanks to Lemma \ref{lem-grr}, $K^y \in \cac_2^{2\ga}([0,1];\cb_p)$ a.s.

\smallskip

\noindent
Finally, for $J^y$, we first lean on the decomposition (\ref{dec-delha-j-n}) of $\delha J^y$ to assert that $\delha J^y\in \cac_3^{\ga+\eta}([0,1];\cb_p)$ a.s. Then we appeal to the expression of $J^y$ that we have exhibited in the proof of Proposition \ref{prop:cas-regu}, namely $J^y_{ts}=\int_s^tS_{tu} \, dx^i_u \, M^i_{us}$ with $M^i$ given by (\ref{exp-m-i}), to show that $E\lc \norm{J^y_{ts}}_{\cb_p}^q \rc \leq c \lln t-s \rrn^{q(\frac{1}{2}+\eta)}$. Together with Lemma \ref{lem-grr}, these results clearly provide the expected regularity, i.e., $J^y \in \cac_2^\mu([0,1];\cb_p)$ a.s, with $\mu=\ga+\eta >1$. 

\smallskip

\noindent
The control of the regularity of $J^y$ as a process with values in $\cb_{\eta,p}$ stems from the same reasoning. Indeed, we first deduce from (\ref{decompo-delha-j-n-2}) that $\delha J^y \in \cac_3^{\ga}([0,1];\cb_{\eta,p})$ a.s, since for instance $\norm{X^{x,i}_{tu}f_i(y_u)}_{\cb_{\eta,p}} \leq c_{x,f,y} \lln t-u \rrn^\ga$ and
$$\norm{X^{x,i}_{tu} (\der x^j)_{us} F_{ij}(y_s)}_{\cb_{\eta,p}} \leq c_x \lln t-s \rrn^{2\ga-(\eta-\frac{1}{2})} \norm{F_{ij}(y_u)}_{\cb_{1/2,p}} \leq c_{x,f,y} \lln t-s \rrn^\ga.$$
We can then write $J^y$ as $J^y_{ts}=\int_s^t S_{tu} \, dx^i_u \, \der(f_i(y))_{us}-X^{xx,ij}_{ts}F_{ij}(y_s)$ to easily derive $E\lc \norm{J^y_{ts}}_{\cb_{\eta,p}}^q \rc \leq c_{x,f,y} \lln t-s \rrn^{q/2}$, and hence $J^y \in \cac_2^\ga([0,1];\cb_{\eta,p})$ a.s.

\end{proof}

\

\textbf{Acknowledgement}

\

I am very grateful to an anonymous reviewer for his/her careful reading and useful suggestions.

\

\end{document}